\documentclass[12pt]{article}

\usepackage[utf8]{inputenc} %% (default option in LaTeX since 2018)

\usepackage{amsmath,amsthm,cite,enumitem,mathtools}
\usepackage[osf]{newtxtext} % osf for text, not math
\usepackage[vvarbb]{newtxmath}
\usepackage[cal=euler]{mathalfa} % mathcal
\usepackage{bm} % bold math fonts load last

\usepackage[margin=3cm]{geometry}

\usepackage{tikz}
\usetikzlibrary{arrows.meta,calc,cd,decorations.pathmorphing, 
                decorations.markings,matrix}
\tikzcdset{arrow style=tikz}
\tikzset{cd/.style=matrix of math nodes}

\usepackage[colorlinks]{hyperref} 
\hypersetup{linkcolor = blue, citecolor = magenta, urlcolor = blue}

\title{Metaplectic and spin representations:\\ a parallel treatment}

\author{Adrián J. Naranjo-Alvarado and Joseph C. Várilly%
\thanks{Email: \texttt{adrian.naranjoalvarado@ucr.ac.cr},\enspace
\texttt{joseph.varilly@ucr.ac.cr}}
\\[12pt]
{\footnotesize 
Escuela de Matemática, Universidad de Costa Rica,
San José 11501, Costa Rica}}

\date{\today}

\DeclareMathOperator{\Arg}{Arg}     %% principal argument
\renewcommand{\det}{\operatorname{det}} %% determinant (nonPlain)
\DeclareMathOperator{\Dom}{Dom}     %% operator domain
\DeclareMathOperator{\End}{End}     %% endomorphisms
\DeclareMathOperator{\im}{im}       %% image space
\DeclareMathOperator{\linspan}{span} %% linear span
\DeclareMathOperator{\per}{per}     %% permanent
\DeclareMathOperator{\Pf}{Pf}       %% a Pfaffian
\newcommand{\al}{\alpha}            %% short for \alpha
\newcommand{\bt}{\beta}             %% short for \beta
\newcommand{\dl}{\delta}            %% short for (Dirac's) delta
\newcommand{\eps}{\varepsilon}      %% short for \varepsilon
\newcommand{\Ga}{\Gamma}            %% short for \Gamma
\newcommand{\ga}{\gamma}            %% short for \gamma
\newcommand{\La}{\Lambda}           %% short for \Lambda
\newcommand{\la}{\lambda}           %% short for \lambda
\newcommand{\Om}{\Omega}            %% short for \Omega
\newcommand{\sg}{\sigma}            %% short for \sigma
\newcommand{\vf}{\varphi}           %% short for \varphi
\newcommand{\ze}{\zeta}             %% short for \zeta

\newcommand{\bC}{{\mathbb{C}}}      %% complex numbers
\newcommand{\bCl}{\operatorname{\bC\ell}} %% complex Clifford algebra
\newcommand{\bR}{{\mathbb{R}}}      %% real numbers
\newcommand{\bT}{\mathbb{T}}        %% circle as a group
\newcommand{\bZ}{{\mathbb{Z}}}      %% integers

\newcommand{\bos}{\mathrm{b}}       %% bosonic label
\newcommand{\Cl}{\mathrm{C}\ell}    %% real Clifford algebra
\newcommand{\even}{\mathrm{even}}   %% even-degree forms
\newcommand{\fer}{\mathrm{f}}       %% fermionic label
\newcommand{\rGL}{\mathrm{GL}}      %% general linear group
\newcommand{\rMp}{\mathrm{Mp}}      %% metaplectic group
\newcommand{\rMpc}{\mathrm{Mp}^\mathrm{c}} %% metaplectic group
\newcommand{\rO}{\mathrm{O}}        %% orthogonal group
\newcommand{\rSk}{\mathrm{Sk}}      %% skewsymmetric operators
\newcommand{\rSO}{\mathrm{SO}}      %% special orthogonal group 
\newcommand{\rSp}{\mathrm{Sp}}      %% symplectic group
\newcommand{\rSpin}{\mathrm{Spin}}  %% spin group
\newcommand{\rSpinc}{\mathrm{Spin}^\mathrm{c}} %% spin-c group
\newcommand{\rU}{\mathrm{U}}        %% unitary group
\newcommand{\spinc}{spin$^\mathrm{c}$} %% in text: \spinc\

\newcommand{\sB}{\mathcal{B}}       %% bosonic Fock space
\newcommand{\sD}{\mathcal{D}}       %% Siegel disk
\newcommand{\sF}{\mathcal{F}}       %% fermionic Fock space
\newcommand{\sJ}{\mathcal{J}}       %% complex structures
\newcommand{\sU}{\mathcal{U}}       %% unitary group of operators
\newcommand{\sW}{\mathcal{W}}       %% polarizations

\renewcommand{\geq}{\geqslant}      %% (variant \geq)
\newcommand{\isom}{\simeq}          %% isomorphism
\renewcommand{\leq}{\leqslant}      %% (variant \leq)
\newcommand{\less}{\setminus}       %% set difference
\newcommand{\ox}{\otimes}           %% tensor product
\newcommand{\rtri}{\blacktriangleright} %% right-facing triangle
\newcommand{\That}{\widehat T}      %% \That_g = T_{g^{-1}}
\renewcommand{\v}{\vee}             %% symmetric product
\newcommand{\w}{\wedge}             %% exterior product
\newcommand{\x}{\times}             %% cartesian product
\newcommand{\7}{\dagger}            %% hermitian conjugate
\newcommand{\8}{\bullet}            %% anonymous degree
\renewcommand{\.}{\cdot}            %% scalar product
\renewcommand{\:}{\colon}           %% colon in `f \: A \to B'

\newcommand{\hookto}{\hookrightarrow} %% inclusion arrw
\newcommand{\otto}{\leftrightarrow} %% bijection

\newcommand{\ovl}{\overline}        %% short for \overline
\newcommand{\wb}{\overline}         %% short for \overline
\newcommand{\oxyox}{\otimes\cdots\otimes} %% repeated tensor product
\newcommand{\stroke}{\mathbin|}     %% dividing line
\newcommand{\vyv}{\vee\cdots\vee}     %% repeated symmetric product
\newcommand{\wyw}{\wedge\cdots\wedge} %% repeated exterior product
\newcommand{\marker}{\vspace{6pt}\noindent{$\rtri$}\enspace} %% topic

\newcommand{\half}{{\mathchoice{\thalf}{\thalf}{\shalf}{\shalf}}}
\newcommand{\ihalf}{{\mathchoice{\tihalf}{\tihalf}{\sihalf}{\sihalf}}}
\newcommand{\quarter}{{\mathchoice{\tcuar}{\tcuar}{\scuar}{\scuar}}}
\newcommand{\scuar}{{\scriptstyle\frac{1}{4}}} %% tiny* fraction 1/4
\newcommand{\shalf}{{\scriptstyle\frac{1}{2}}} %% tiny* fraction 1/2
\newcommand{\sihalf}{\scriptstyle\frac{i}{2}}  %% tiny* fraction i/2
\newcommand{\tcuar}{\tfrac{1}{4}}     %% small* fraction  1/4
\newcommand{\thalf}{\tfrac{1}{2}}     %% small* fraction  1/2
\newcommand{\tihalf}{\tfrac{i}{2}}    %% small* fraction  i/2

\newcommand{\miss}[1]{\widehat{#1}}   %% omit one symbol
\newcommand{\set}[1]{\{\,#1\,\}}      %% set notation
\newcommand{\word}[1]{\quad\text{#1}\quad} %% well-spaced word(s)

\newcommand{\braket}[2]{\langle#1\stroke#2\rangle} %% <p|q>
\newcommand{\brakett}[2]{\langle\!\langle#1\stroke#2\rangle\!\rangle}
\newcommand{\ketbra}[2]{\lvert#1\rangle\langle#2\rvert} %% |z><w|
\newcommand{\ddto}[1]{\frac{d}{d#1}\biggr|_{#1=0}} %% derivative at 0
\newcommand{\twobytwo}[4]{\begin{pmatrix} %% 2 x 2 matrix
       #1 & #2 \\ #3 & #4 \end{pmatrix}}

\numberwithin{equation}{section}     %% \eqref -> (2.71)
\theoremstyle{plain}
\newtheorem{thm}{Theorem}[section]    %% Theorem 3.14
\newtheorem{prop}[thm]{Proposition}   %% Proposition 3.15
\newtheorem{lema}[thm]{Lemma}         %% Lemma 3.16
\newtheorem{corl}[thm]{Corollary}     %% Corollary 3.17

\theoremstyle{definition}
\newtheorem{defn}[thm]{Definition}    %% Definition 3.18

\theoremstyle{remark}
\newtheorem*{notn}{Notation}          %% Notation

\makeatletter
\renewcommand{\section}{\@startsection{section}{1}{\z@}%
                        {-3.5ex \@plus -1ex \@minus -.2ex}%
                        {2.3ex \@plus.2ex}%
                        {\normalfont\large\bfseries}}
\renewcommand{\subsection}{\@startsection{subsection}{2}{\z@}%
                        {-3.25ex \@plus -1ex \@minus -.2ex}%
                        {1.5ex \@plus .2ex}%
                        {\normalfont\normalsize\bfseries}}
\renewcommand{\subsubsection}{\@startsection{subsubsection}{3}{\z@}%
                        {-3.25ex \@plus -1ex \@minus -.2ex}%
                        {1.5ex \@plus .2ex}%
                        {\normalfont\small\bfseries}}
\renewcommand{\@dotsep}{200} %% suppress dots in Contents
\makeatother

\newcommand{\qnedsymbol}{$\boxminus$} %% no proof
\newcommand{\qefsymbol}{$\lozenge$}   %% end of definition

\DeclareRobustCommand{\qned}{\ifmmode
  \else \leavevmode\unskip\penalty9999 \hbox{}\nobreak\hfill \fi
  \quad\hbox{\qnedsymbol}}

\DeclareRobustCommand{\qef}{\ifmmode
  \else \leavevmode\unskip\penalty9999 \hbox{}\nobreak\hfill \fi
  \quad\hbox{\qefsymbol}}

\newcommand{\hideqed}{\renewcommand{\qed}{}} %% suprimir `\qed'
\newcommand{\hideqef}{\renewcommand{\qef}{}} %% suprimir `\qef'

\begin{document}

\maketitle

\begin{abstract}  

The analogies between symplectic and orthogonal groups, regarded as
symmetries of real bilinear forms, are manifest in their (metaplectic
and spin) projective representations. In finite dimensions, those are
true representations of doubly covering groups; but one may also use 
group extensions by a circle. Here we lay out a parallel treatment
of the Mp$^\mathrm{c}$ and Spin$^\mathrm{c}$ covering groups, acting 
on the respective Fock spaces by permuting certain Gaussian vectors.
The cocycles of these extensions exhibit interesting similarities.
\end{abstract}

%============================================

\section{Introduction} %% 1
\label{sec:intro}

The well-known thesis of David Shale~\cite{Shale62}, and the follow-up
paper with Stinespring~\cite{ShaleS65}, established that the linear
symmetries of free boson fields or free fermion fields exhibit a
common feature: such symmetries are implementable on Fock space if and
only if their antilinear parts (in a sense made precise below) are
Hilbert--Schmidt operators. In more modern
terminology~\cite{PressleyS86}, these symmetries belong to the
\textit{restricted} general linear group of the one-particle space.
These one-particle spaces come equipped with a real bilinear form
which is skewsymmetric in the boson case but symmetric in the fermion
case: the corresponding symmetries belong to the restricted symplectic
group or the restricted orthogonal group, respectively.

With finitely many degrees of freedom, the one-particle spaces have
finite even dimension, the Hilbert--Schmidt condition becomes trivial,
and one is left with the ordinary Lie groups $\rSp(2m,\bR)$ and
$\rSO(2m)$, respectively. Both of these have ``double-valued''
projective representations, which are tantamount to true
representations of their double covering groups. Such double covers
are not sufficient for infinite-dimensional one-particle spaces, where
different extensions are needed: there, the appropriate central
extensions are not by~$\{\pm 1\}$ but by the circle group $\rU(1)$.
Naturally, one may also do this in finite dimensions: in the fermionic
case, the extended group is $\rSpinc(2m)$, the structure group of
``\spinc\ structures'' (if~any) on a Riemannian manifold. The
analogous concept of ``$\rMpc$ structures'' on a symplectic manifold
is perhaps less familiar, but has been developed in a similar
fashion~\cite{RobinsonR89}.

These projective representations act on Fock spaces. For boson fields,
the Fock space is best treated in the ``complex wave representation''
proposed by Irving Segal \cite{Segal60, Segal78}, and developed in
detail by Bargmann \cite{Bargmann61, Bargmann67}, as a concrete
Hilbert space of holomorphic or antiholomorphic entire functions.
Fermions, on the other hand, need a space of ``anticommuting (or
Grassmann) variables'': their Fock spaces appear as exterior algebras.

These symmetry groups act on the Fock spaces by permuting a
distinguished set of vectors (up to phase factors), that we call
Gaussian vectors, manifested by quadratic exponentials. (The recent
work by Neumaier and Ghanni Farashahi \cite{NeumaierGF22} places this
permutation action in a more abstract setting, for bosons at any
rate.) The main task is to identify the $2$-cocycles or factor systems
of their phases. The Shale--Stinespring method leads to
$\rU(1)$-valued cocycles that show striking resemblances as well as
subtle differences between the bosonic and fermionic cases. To
simplify the comparisons, here we limit ourselves to finitely many
degrees of freedom. The one-particle space in each case will be a real
vector space~$V$ of finite dimension~$2m$, equipped with a
nondegenerate bilinear form, either skewsymmetric or symmetric.
Canonical quantization requires the presence of a compatible linear
complex structure, so that $V$ becomes an $m$-dimensional complex
Hilbert space, whereby the Fock-space construction (i.e., second
quantization) can proceed.

The $\rU(1)$-metaplectic representation in this sense first appears,
to our knowledge, in a short article by Vergne~\cite{Vergne77}, which
improves notably the standard case-by-case construction of the
double-valued metaplectic representation found in many textbooks,
e.g., \cite{Folland89, Wallach18}. An adaptation of this presentation
to fermionic free fields appeared in~\cite{Rhea} by J.\,M.
Gracia-Bondía and one of us; this is retrieved in Chapter~6
of~\cite{Polaris}. An exposition of both constructions forms the basis
of the first author's M.\,Sc. thesis~\cite{Kamuk}.

\bigskip

The plan of the article is as follows. In
Section~\ref{sec:linear-algebra} we develop the linear algebra of
one-particle spaces with finite even dimension. Emphasis is placed on
a particular parametrization of the symplectic group (for bosons),
which underlies the monograph~\cite{RobinsonR89}, and a similar
parametrization of the special orthogonal group (for fermions).
Section~\ref{sec:Fock-spaces} brings in the corresponding Fock spaces,
focusing attention on their Gaussian vectors. Each type of Fock space
carries a special projective representation of the symmetry group,
which is implemented mainly by permuting the Gaussian vectors; these
are laid out in Section~\ref{sec:meta-spin}. We develop both
representations in a parallel fashion, which allows to compare their
cocycles so as to reveal their common features. In
Appendix~\ref{app:pin-repn}, we deal with reflections, extending the
spin representation to the full orthogonal group.

\section{Symplectic and orthogonal vector spaces} %% 2
\label{sec:linear-algebra}

Our starting point is a real vector space $V$ with even dimension
$n = 2m$, equipped with one of the following bilinear forms.

\begin{description}
\item[Case B]
Fix a nondegenerate \textit{skewsymmetric} (i.e., symplectic) bilinear
form $s \: V \x V \to \bR$. We call $(V,s)$ a \textit{symplectic
vector space}.
\item[Case F]
Fix a positive definite \textit{symmetric} bilinear form
$d \: V \x V \to \bR$, i.e., a real scalar product on~$V$. We call
$(V,d)$ an \textit{orthogonal vector space}.
\end{description}

Nondegeneracy of~$s$ demands that $\dim_\bR V$ be~even. This is not
mandatory for $(V,d)$, but we keep the even dimension for comparison's
sake.

In either case, one can introduce complex structures compatible with 
the given bilinear form; in Case~B, there is an extra requirement of 
positivity.

\begin{defn} % 2.1
\label{df:complex-structure}
A (linear) \textit{complex structure} on~$V$ is a real-linear operator 
$J \in \End_\bR(V)$ such that $J^2 = -1$; we say $J$ is 
\textit{compatible} if either of the following conditions holds:
\begin{enumerate}
\item % (a)
A \textit{positive complex structure} on $(V,s)$ must also satisfy:
\begin{align*}
s(Ju, Jv) &= s(u, v)  \word{for all} u,v \in V,
\\
s(u, Ju) &> 0  \word{for} u \in V, \ u \neq 0.
\end{align*}
\item % (b)
An \textit{orthogonal complex structure} on $(V,d)$ must also satisfy:
$$
d(Ju, Jv) = d(u, v)  \word{for all} u,v \in V.
$$
\end{enumerate}

In either case, a compatible complex structure gives rise to a second
bilinear form on~$V$:
\begin{itemize}
\item
On $(V,s)$ with a positive~$J$, define $d(u,v) := s(u,Jv)$.
\item
On $(V,d)$ with an orthogonal~$J$, define $s(u,v) := d(Ju,v)$.
\qef
\end{itemize}
\hideqef
\end{defn}

It is easily checked, using $J^2 = -1$, that the new~$d$ is symmetric
and positive definite; and that the new~$s$ is symplectic. Moreover,
a compatible $J$ now satisfies \textit{both} conditions (a) and~(b)
of Definition~\ref{df:complex-structure}.

In both cases, the transpose $A^t$ of a real-linear map 
$A \in \End_\bR(V)$ is taken with respect to the real scalar
product~$d$, i.e., $d(u, A^t v) := d(Au, v)$. It is then immediate
that $J^t = J^{-1} = -J$. 

\medskip 

The point of having both $d$ and~$s$ available (mediated by~$J$) is
that these become the real and imaginary parts of a complex scalar
product. Note first that any complex structure confers on~$V$ the
status of a \textit{complex} vector space of dimension $m = \half n$,
just by setting
$$
(a + bi)v := av + bJv  \word{for} a,b \in \bR; \ v \in V.
$$  
Given either $(V,s)$ or~$(V,d)$ together with $J$ as in
Definition~\ref{df:complex-structure}, define for $u,v \in V$:
\begin{equation}
\braket{u}{v}
:= d(u,v) + is(u,v) = s(u,Jv) + is(u,v) = d(u,v) + id(Ju,v).
\label{eq:scalar-product} % (2.1)
\end{equation}
Notice that this complex scalar product $\braket{u}{v}$ is linear in
the second argument~$v$ and antilinear in the first argument~$u$ (the
usual convention in physics). From~\eqref{eq:scalar-product} it is
clear that
\begin{equation}
\braket{u}{Jv} = i\,\braket{u}{v}   \word{and}
\braket{Ju}{v} = -i\,\braket{u}{v}  \quad\text{for all } u,v \in V.
\label{eq:effect-of-J} % (2.2)
\end{equation}

From now on, $V$ will denote the $m$-dimensional complex Hilbert space
with the scalar product~\eqref{eq:scalar-product} (sometimes writing
$V_J$ for clarity). It remains useful to consider real-linear
operators on~$V$. Such an $R \in \End_\bR(V)$ is called
\textit{linear} (i.e., $\bC$-linear) if $RJ = JR$; whereas
$S \in \End_\bR(V)$ is \textit{antilinear} if $SJ = - JS$.

Of course, any $\bR$-linear $T \in \End_\bR(V)$ is uniquely a sum
$T = T_1 + T_2$ of $\bC$-linear and antilinear parts, by setting
$T_1 := \half(T - JTJ)$ and $T_2 := \half(T + JTJ)$.

\subsection{Symmetry groups} % 2.1
\label{ssc:symmetry-groups}

The symmetry groups of~$V$ consist of those $\bR$-linear maps that
leave $s$ or~$d$ invariant. These are the \textit{symplectic group}
$\rSp(V) = \rSp(V,s)$, of maps $g$ satisfying
$s(gu, gv) \equiv s(u,v)$; and respectively the \textit{orthogonal
group} $\rO(V) = \rO(V,d)$ of maps $g$ such that
$d(gu, gv) \equiv d(u,v)$, for all $u,v \in V$. Choosing a real basis
for~$V$ yields isomorphisms $\rSp(V) \isom \rSp(2m,\bR)$ and
$\rO(V) \isom \rO(2m)$. Thus the Lie group $\rSp(V)$ is noncompact and
connected, whereas $\rO(V)$ is compact and has two connected
components. Now, $\det g = +1$ for $g \in \rSp(V)$, while in $\rO(V)$
one finds $\det g = \pm 1$, the neutral component being
$\rSO(V) := \set{g \in \rO(V) : \det g = +1}$.

Given a complex structure $J$ satisfying either (a) or~(b) of
Definition~\ref{df:complex-structure}, so that the complex scalar
product~\eqref{eq:scalar-product} is defined, the intersection of both
symmetry groups is the corresponding \textit{unitary group}:
\begin{equation}
\rU_J(V) := \rSp(V) \cap \rO(V) = \rSp(V) \cap \rSO(V).
\label{eq:unitary-group} % (2.3)
\end{equation}
Remark that $\rU_J(V) \subset \rSO(V)$ since $\rU_J(V) \isom \rU(m)$
is connected.

It is immediate that $g \in \rO(V)$ if and only if $g^t g = 1 = gg^t$;
and on recalling the conventions of 
Definition~\ref{df:complex-structure}, that $g \in \rSp(V)$ if and
only if $g^t J g = J = g J g^t$. In both cases, it is helpful to write
$g = p_g + q_g$ (uniquely) as the sum of a linear map $p_g$ and an
antilinear map $q_g$, namely:
$$
p_g := \half(g - JgJ) \word{and} q_g := \half(g + JgJ).
$$
Notice that $g$ will be (complex) linear and thus $q_g = 0$, if and
only if $g \in \rU_J(V)$: one checks that
$\braket{gv}{gu} \equiv \braket{v}{u}$ if and only if $Jg = gJ$. We
summarize as follows.

\begin{prop} % 2.2
\label{pr:p-q-relations}
If $g = p_g + q_g$ expresses $g \in \rSp(V)$ or $g \in \rO(V)$ as a 
sum of its linear and antilinear parts, then
\begin{equation}
p_g^t p_g \mp q_g^t q_g = 1 = p_g p_g^t \mp q_g q_g^t, \quad
p_g^t q_g = \pm q_g^t p_g,  \quad  p_g q_g^t = \pm q_g p_g^t
\label{eq:p-q-relations} % (2.4)
\end{equation}
where each upper sign goes with $g \in \rSp(V)$ and the lower sign
with $g \in \rO(V)$.
\end{prop}

The proof is a straightforward calculation (often written in
matrix~form). Notice that $p_g^t p_g$ and $q_g^t q_g$ are symmetric
and positive definite operators on the Euclidean space $(V,d)$, but
the signs in the first relation of~\eqref{eq:p-q-relations} point to a
crucial difference between the two groups:
\begin{itemize}
\item
If $g \in \rSp(V)$, then $p_g^t p_g = 1 + q_g^t q_g \geq 1$, so that
the linear part $p_g$ is invertible.
\item
If $g \in \rO(V)$, then $p_g^t p_g + q_g^t q_g = 1$, so that
$0 \leq p_g^t p_g \leq 1$, but $p_g$ may or may not be invertible.
\end{itemize}

Here the relation $A \geq B$ among symmetric linear operators
($A^t = A$, $B^t = B$) means that $A - B$ is positive semidefinite;
and one writes $A > B$ if $A - B$ is positive definite.

\marker
The $(p_g,q_g)$ parametrization of the symmetry groups is widely 
used, but for our purposes an alternative parametrization is better.
Following Robinson and Rawnsley \cite{RobinsonR89}, though with 
slightly different conventions, we replace the antilinear part $q_g$
with the following antilinear operator.

\begin{notn}
For either $g \in \rSp(V)$ or $g \in \rO(V)$ with $p_g$ invertible,
define
$$
T_g := q_g p_g^{-1} \,.
$$
It turns out~\cite{Rhea, Araki87} that $p_g$ is invertible only when
$g \in \rSO(V)$, the neutral component of~$\rO(V)$. In the Lie group
$\rSO(V)$, denote by $\rSO_*(V)$ the complement of the closed
submanifold wherein $\det p_g = 0$. Then $T_g$ is defined in this
subset.

Since $g^{-1} = g^t$ in the orthogonal case, it is clear that
$\rSO_*(V)$ is invariant under inversion $g \mapsto g^{-1}$. We use
the abbreviation $\That_g := T_{g^{-1}}$ for convenience.
\end{notn}

The following three lemmas, that develop some algebraic relations
among these group components, are taken from~\cite{Kamuk}.

\begin{lema} % 2.3
\label{lm:T-hat}
If either $g \in \rSp(V)$ or $g \in \rSO_*(V)$, then
$\That_g = - p_g^{-1} T_g p_g$\,.
\end{lema}

\begin{proof}
First notice that $(1 + T_g) p_g = p_g + q_g = g$. For 
$g \in \rSp(V)$, 
$$
g^{-1} = -J g^t J = -J (p_g^t + q_g^t) J = p_g^t - q_g^t
\implies  p_{g^{-1}} = p_g^t\,, \quad q_{g^{-1}} = - q_g^t\,,
$$
whereas $p_{g^{-1}} = p_g^t$ but $q_{g^{-1}} = + q_g^t$ for
$g \in \rSO_*(V)$. Thus, in both cases, $gg^{-1} = 1$ entails
$$
1 = (p_g + T_g p_g)(p_g^t + \That_g p_g^t)
= p_g p_g^t + p_g \That_g p_g^t
+ T_g p_g p_g^t + T_g p_g \That_g p_g^t.
$$
The antilinear part of this equality is
$p_g \That_g p_g^t + T_g p_g p_g^t = 0$, so that 
$p_g \That_g = - T_g p_g$ and thus $\That_g = - p_g^{-1} T_g p_g$
since $p_g$ and $p_g^t$ are both invertible.
\end{proof}

\begin{lema} % 2.4
\label{lm:symmetry-of-T}
If $g \in \rSp(V)$, then $T_g$ is \emph{symmetric} and satisfies
$0 < 1 - T_g^2 \leq 1$. By contrast, if $g \in \rSO_*(V)$, then
$T_g$ is \emph{skewsymmetric} and satisfies $1 - T_g^2 \geq 1$.
\end{lema}

\begin{proof}
It follows directly from the third equality
in~\eqref{eq:p-q-relations} that 
$$
T_g^t = p_g^{-t} q_g^t = \pm q_g p_g^{-1} = \pm T_g \,,
$$
i.e., $T_g^t = T_g$ if $g \in \rSp(V)$, whereas 
$T_g^t = - T_g$ if $g \in \rSO_*(V)$. 

It is useful to express these properties of the \textit{antilinear}
operator $T_g$ in terms of the scalar product
\eqref{eq:scalar-product} by:
\begin{equation}
\braket{u}{T_g v} = \pm \braket{v}{T_g u} \word{for all} u,v \in V,
\label{eq:symmetry-of-T} % (2.5)
\end{equation}
with sign $+$ for symmetry, $-$ for skewsymmetry. 

\goodbreak %%%

Since $g = p_g + q_g = (1 + T_g) p_g = p_g (1 - \That_g)$ using the 
previous proof, the exchange $g \otto g^{-1}$ yields
$g^{-1} = p_g^t (1 - T_g)$, whereby
\begin{equation}
1 = g^{-1} g = p_g^t (1 - T_g) (1 + T_g) p_g = p_g^t (1 - T_g^2) p_g.
\label{eq:fancy-formula} % (2.6)
\end{equation}
That implies $1 - T_g^2 = (p_g p_g^t)^{-1}$, which is positive
semidefinite and invertible, thus positive definite. On replacing $g$ 
by~$g^{-1}$, we also get $1 = p_g (1 - \That_g^2) p_g^t$ and so
$1 - \That_g^2 = (p_g^t p_g)^{-1}$.

In Case~B, it is also true that $T_g^2 = T_g^t T_g$ is positive 
semidefinite, yielding $1 - T_g^2 \leq 1$. But in Case~F, one sees 
that $T_g^2 = - T_g^t T_g$ is negative semidefinite, and so 
$1 - T_g^2 \geq 1$.
\end{proof}

We shall need the composition formulas for $\rSp(V)$ and $\rSO_*(V)$
in terms of the $(p,T)$-parametrization. In Case~F, these only apply
to $g,h \in \rSO_*(V)$ only if $gh \in \rSO_*(V)$ also.

\begin{lema} % 2.5
\label{lm:compositions}
If $g,h \in \rSp(V)$ or if $g,h,gh \in \rO_*(V)$, then
\begin{subequations}
\label{eq:compositions} % (2.7)
\begin{align}
p_{gh} & = p_g (1 - \That_g T_h) p_h,
\label{eq:p-compositions} % (2.7a)
\\
T_{gh} & = T_g + p_g^{-t}\, T_h(1 - \That_g T_h)^{-1} p_g^{-1}.
\label{eq:T-compositions} % (2.7b)
\end{align}
\end{subequations}
\end{lema}

\begin{proof}
First of all, notice that
\begin{align*}
p_{gh} + T_{gh} p_{gh} = gh &= (1 + T_g)p_g (1 + T_h)p_h
= p_g(1 - \That_g)(1 + T_h)p_h
\\
&= p_g p_h + p_g T_h p_h - p_g \That_g p_h - p_g \That_g T_h p_h.
\end{align*}
The linear parts of both sides yield~\eqref{eq:p-compositions}:
$$
p_{gh} = p_g p_h - p_g \That_g T_h p_h = p_g(1 - \That_g T_h)p_h.
$$

Next, from the second equality in~\eqref{eq:p-q-relations} it follows 
that
$$
p_g + q_g \That_g = p_g + q_g(\mp q_g^t p_g^{-t}) 
= (p_g p_g^t \mp q_g q_g^t)p_g^{-t} = p_g^{-t}.
$$
Now $gh = (p_g + q_g)(p_h + T_h p_h)$ also entails 
$p_{gh} = (p_g + q_g T_h) p_h$ and $q_{gh} = (q_g + p_g T_h) p_h$,
whereby
\begin{equation}
T_{gh} = q_{gh} p_{gh}^{-1} = (q_g + p_g T_h)(p_g + q_g T_h)^{-1}.
\label{eq:T-composition-bis} % (2.8)
\end{equation}
Using Lemma~\ref{lm:T-hat}, one finds that
$$
p_g + q_g T_h = p_g + T_g p_g T_h  = p_g(1 - \That_g T_h),
$$
and the desired result emerges from~\eqref{eq:T-composition-bis}
and $(p_g + q_g \That_g)p_g^t = p_q p_g^t \mp q_g q_g^t = 1$:
\begin{align*}
T_{gh} &= q_g p_g^{-1} + [q_g + p_gT_h - q_g(1 - \That_g T_h)]
(1 - \That_g T_h)^{-1} p_g^{-1}
\\
&= T_g + p_g^{-t}\, T_h (1 - \That_g T_h)^{-1} p_g^{-1}.
\tag*{\qed}
\end{align*}
\hideqed
\end{proof}

\subsection{Homogeneous spaces and polarizations} % 2.2
\label{ssc:group-actions}

The conjugation action $J \mapsto g J g^{-1}$ of the symmetry group
permutes the compatible complex structures. The isotropy subgroup
of~$J$ is the unitary group $\rU_J(V)$, thus the set of compatible
complex structures $\sJ(V)$ may be identified with a homogeneous
space. In Case~B, this is $\rSp(V)/\rU_J(V)$, a noncompact connected
manifold of dimension $m(m + 1)$. In Case~F, it is $\rO(V)/\rU_J(V)$,
a compact two-component manifold of dimension $m(m - 1)$.

Another way to count the compatible complex structures is to example 
the \textit{polarizations} of the complex vector space
$$
V_\bC := V \ox_\bR \bC  \isom  V \oplus iV.
$$
We adopt the notations $w = u + iv$ and $\bar w = u - iv$ in~$V_\bC$,
for $u,v \in V$. The real bilinear forms $s$, $d$ on~$V$ determine
complex \textit{bilinear} (not sesquilinear) forms on~$V_\bC$ in the 
obvious way; and similarly any $A \in \End_\bR(V)$ amplifies to
$A \in \End_\bC(V_\bC)$. Consider the complex subspaces
$$
W_{\pm J} := \set{\half(u \mp iJu) : u \in V} \leq V_\bC \,.
$$
Since $J(u \mp iJu) = \pm i(u \mp iJu)$, $W_J$ and $W_{-J}$ are the
$(\pm i)$-eigenspaces of the amplified $J \in \End_\bC(V_\bC)$. This
entails that $V_\bC = W_J \oplus W_{-J}$ as a direct sum of two
complex subspaces of dimension~$m$. These subspaces are also
\textit{isotropic} for both $s$ and~$d$. For instance,
$$
s(u - iJu, v - iJv) = s(u,v) - i\,d(u,v) + i\,d(v,u) - s(u,v) = 0,
$$
and likewise for the other cases. Also, $W_{-J} = \wb{W_J}$, where the
bar refers to the conjugation operator $w \mapsto \bar w$ on~$V_\bC$.

\begin{defn} % 2.6
\label{df:polarization}
A complex subspace $W \leq V_\bC$ is called a \textit{polarization} 
of~$V_\bC$ if
$$
W \cap \wb{W} = \{0\}  \word{and}  W \oplus \wb{W} = V_\bC \,.
$$
We shall only consider polarizations that are \textit{isotropic} with
respect to the bilinear form~$s$, in Case~B; or isotropic with respect
to~$d$, in Case~F. If $W$ is isotropic for~$s$ [or~$d$], then so too
is the complementary polarization~$\wb{W}$.
\end{defn}

\marker
Given a polarization that is isotropic for either $s$ or~$d$, write
$w \in W$ uniquely as $w = u - iv$ with $u,v \in V$; then the
real-linear maps $w \mapsto u$, $w \mapsto v$ are one-to-one, and have
continuous inverses since the scalar product~\eqref{eq:scalar-product}
extends to~$V_\bC$ as a \textit{sesquilinear} form such that so that
$\braket{w}{w} = \braket{u}{u} + \braket{v}{v}$. The composite map
$u \mapsto w \mapsto v$ is thus an invertible $\bR$-linear operator
$J_W \in \End_\bR(V)$. In this way,
$$
W = \set{w = \half(u - i J_W u) : u \in V}.
$$
In Case~B, when $w_1,w_2 \in W$, the isotropy relation
$s(w_1, w_2) = 0$ expands to:
$$
s(u_1, u_2) - s(J_W u_1, J_W u_2)
+ i\,s(J_W u_1, u_2) - i\,s(u_1, J_W u_2) = 0.
$$
The real and imaginary parts of this equality show that $J_W$ is
symplectic and then that $J_W^2 = -1$. By a similar calculation in
Case~F, $d(w_1, w_2) = 0$ implies that $J_W$ is orthogonal with
$J_W^2 = -1$.

In Case~B, one may say that a polarization $W$ is \textit{positive}
if $s(u, J_W u) > 0$ for $u \neq 0$ in~$V$. In that case, the 
sesquilinear form $r(w_1, w_2) := 2i\,s(\bar w_1,w_2)$ satisfies
$$
r(w,w) = \ihalf\,s(u + i J_W u, u - i J_W u) = s(u, J_W u),
$$
and so $r$~is positive definite on~$W$ and is negative definite
on~$\wb{W}$. Now, since $r$ does not depend directly on the complex
structure, this also implies that $W' \cap \wb{W} = \{0\}$ whenever
$W$ and~$W'$ are positive polarizations.

We shall call a polarization $W$ \textit{admissible} if either of 
these conditions holds:
\begin{itemize}
\item
in Case~B, if $W$ is isotropic with respective to~$s$ and is positive;
\item
in Case~F, if $W$ is isotropic with respective to~$d$.
\end{itemize}
Denote the set of admissible polarizations by $\sW(V)$, in either
case. Then $W \mapsto J_W$ defines a bijection $\sW(V) \to \sJ(V)$
with the compatible complex structures on~$V$. The inverse bijection 
is $J \mapsto W_J$.

The symmetry group $\rSp(V)$ or~$\rO(V)$ acts on $\sW(V)$ by 
$W \mapsto gW \equiv \set{gw : w \in W}$, and one checks that
$J_{gW} = g\,J_W\,g^{-1}$; therefore, the aforementioned 
bijections are equivariant.

\marker
From now on, we fix a particular compatible complex structure $J$ on
either $(V,s)$ or $(V,d)$. On the complexified space $V_\bC$, one can
define a complex scalar product by
\begin{equation}
\brakett{w_1}{w_2} := 2s(\bar w_1, Jw_2) = 2d(\bar w_1, w_2). 
\label{eq:W-pairing} % (2.9)
\end{equation}
If $W$ is an admissible polarization, the direct sum 
$V_\bC = W \oplus \wb{W}$ then becomes an \textit{orthogonal} direct
sum, with the respective orthogonal projectors on these subspaces
\begin{equation}
P_+ := \half(1 - iJ) \word{and} P_- := \half(1 + iJ).
\label{eq:projectors} % (2.10)
\end{equation}
It is readily checked that
$$
\brakett{P_+u}{P_+v} = \braket{u}{v},  \qquad
\brakett{P_-u}{P_-v} = \braket{v}{u} \word{for}  u,v \in V.
$$
A different but equivalent treatment -- for the fermionic case only --
is given by Araki in~\cite{Araki87}, where these $P_\pm$ are called
\textit{basis projections}.

\goodbreak %%%

\marker
A third presentation of these homogeneous spaces uses the antilinear 
maps~$T$ introduced in subsection~\ref{ssc:symmetry-groups}.

In Case~B, one can start with any admissible polarization $W$ and 
recall that $\wb{W} \cap W_J = \{0\}$. Then any $w \in W$ can be 
written uniquely as $w = z_1 + \bar z_2$ with $z_1,z_2 \in W_J$. 
Using the scalar product \eqref{eq:W-pairing} on~$V_\bC$, one finds
$\brakett{w}{w} = \brakett{z_1}{z_1} + \brakett{\bar z_2}{\bar z_2}$,
which entails that $w \mapsto z_1$ and $w \mapsto \bar z_2$ are
injective $\bC$-linear maps with continuous inverses. The composite
mapping
\begin{equation}
T_W \: u_1 \mapsto \half(u_1 - iJu_1) =: z_1 \mapsto w 
\mapsto \bar z_2 := \half(u_2 + iJu_2) \mapsto u_2
\label{eq:T-map} % (2.11)
\end{equation}
is an $\bR$-linear operator on~$V$. The same map applied to~$Ju_1$ 
gives
$$
Ju_1 \mapsto \half(Ju_1 + iu_1) = iz_1 \mapsto iw 
\mapsto i\bar z_2 = -\half(Ju_2 - iu_2) \mapsto - Ju_2,
$$
since the intermediate maps $z_1 \mapsto w \mapsto \bar z_2$ are
$\bC$-linear. Therefore, $T_WJ = -JT_W$, i.e., $T_W$ is 
\textit{antilinear} on~$V$.

Moreover, if $u_1, u'_1 \in V$, with $w' = z'_1 + \bar z'_2 
= \half(u'_1 - iJu'_1) + \half(u'_2 + iJu'_2) \in W$, then the
orthogonality of $\wb{W}$ and~$W$ shows that $T_W$ is
\textit{symmetric}, because
\begin{align}
0 = \Re r(\bar w, w')
&= 2 \Im s(w, w') = 2 \Im s(z_1 + \bar z_2, z'_1 + \bar z'_2)
\nonumber \\
&= \half \Im s(u_1 - iJu_1, u'_2 + iJu'_2)
+ \half \Im s(u_2 + iJu_2, u'_1 - iJu'_1)
\nonumber \\
&= s(u_1, Ju'_2) - s(u_2, Ju'_1) = d(u_1, u'_2) - d(u_2, u'_1)
\nonumber \\
&= d(u_1, T_W u'_1) - d(T_W u_1, u'_1).
\label{eq:TW-symmetry} % (2.12)
\end{align}

Now, the calculation
$$
w = z_1 + \bar z_2 = \half(u_1 - iJu_1 + u_2 + iJu_2)
= \half(1 + T_W)(u_1 - iJu_1) = (1 + T_W) z_1
$$
shows that $1 + T_W \: W_J \to W$ is invertible; indeed, it is the
inverse of the bijection $w \to z_1$. Next,
\begin{align*}
J_W(1 + T_W)(u_1 + iu_2) &= 2 J_W w = 2iw = 2i(z_1 + \bar z_2)
\\
&= (Ju_1 - Ju_2) + i(u_1 + u_2) = J(1 - T_W)u_1 + i(u_1 + u_2),
\end{align*}
and in particular $J_W(1 + T_W)u_1 = J(1 - T_W)u_1$ for any 
$u_1 \in V$. Therefore, the complex structures $J$ and $J_W$ are
related by a \textit{Cayley transformation}:
\begin{equation}
J_W = J(I - T_W)(I + T_W)^{-1},
\label{eq:Cayley-transformation} % (2.13)
\end{equation}
which can be inverted by $T_W = (J - J_W)(J + J_W)^{-1}$.

Moreover, the linear operator $1 - T_W^2$ is positive definite on~$V$,
since for $v \neq 0$:
\begin{align*}
\braket{v}{(1 - T_W^2)v}
&= d(v + T_W v, v - T_W v) = s((1 + T_W)v, J(1 - T_W)v)
\\
&= s((1 + T_W)v, J_W(1 + T_W)v) > 0, 
\end{align*}
because $W$ is a positive polarization and $(1 + T_W)v \neq 0$.

The upshot, in the bosonic case, is that each $T_W$ belongs to the set
$$
\sD(V) := \set{X \in \End_\bR V : XJ = -JX,\ X^t = X,\ 1 - X^2 > 0}.
$$
Lemma~\ref{lm:symmetry-of-T} shows that $T_g \in \sD(V)$ for all
$g \in \rSp(V)$. With a little more work~\cite{Hyperion}, it is
possible to show that $W \otto T_W$ is a bijection between $\sW(V)$
and $\sD(V)$ which is $\rSp(V)$-equivariant. The action of the group
on $\sD(V)$ is given by a variant of the composition formulas
\eqref{eq:T-compositions} and~\eqref{eq:T-composition-bis}:
\begin{equation}
g \. S := T_g + p_g^{-t}\,S(1 - \That_g S)^{-1} p_g^{-1}
= (q_g + p_g S)(p_g + q_g S)^{-1}.
\label{eq:group-action-on-T} % (2.14)
\end{equation}

\marker
Matters are somewhat different in the fermionic case, where the
hypotheses leading to \eqref{eq:T-map} do not always hold.

\begin{prop} % 2.7
\label{pr:nice-kernel}
Given an admissible polarization~$W$ in Case~F, one may introduce the
linear operator $B := \half(1 - J_W J)$ on the Hilbert space $V_\bC$,
whose adjoint is $B^\7 := \half(1 - J J_W)$. Their kernels coincide:
$$
\ker B = \ker B^\7 = (W \cap \wb{W_J}) \oplus (\wb{W} \cap W_J).
$$
\end{prop}

\begin{proof}
The expression for the kernel of~$B$ follows from
\begin{align*}
\ker B &= \set{w \in V_\bC : w = J_W J(w)}
= \set{z \in V_\bC : Jz = - J_Wz}
\\
&= \set{z \in V_\bC : \half(1 \mp iJ_W)z = \half(1 \pm iJ)z}
= \set{z \in V_\bC : P'_\pm z = P_\mp z}
\\
&= (W \cap \wb{W_J}) \oplus (\wb{W} \cap W_J),
\end{align*}
where the projectors $P'_\pm := \half(1 \mp iJ_W)$ have range spaces
$W$ and $\wb{W}$ respectively.

In the same way, $\ker B^\7 = \ker(1 - J J_W) 
= \set{z \in V_\bC : J_Wz = -Jz} = \ker B$ by exchanging $J_W$ 
and~$J$.
\end{proof}

The subspaces $W \cap \wb{W_J}$ and $\wb{W} \cap W_J$ are complex 
conjugates of each other, so they have the same dimension. This 
entails that $\dim(\ker B)$ is \textit{even}. 

Then, of course, $\dim(\wb{W} \cap W_J) = \half \dim(\ker B)$. If this
number is~$0$, then the procedure of Eq.~\eqref{eq:T-map} allows one
to define the antilinear operator $T_W \: u_1 \mapsto u_2$ as before.
The analogue of~\eqref{eq:TW-symmetry} is straightforward:
\begin{align*}
0 = 2 \Re d(w, w') &= 2 \Re d(z_1 + \bar z_2, z'_1 + \bar z'_2)
\\
&= \half \Re d(u_1 - iJu_1, u'_2 + iJu'_2)
+ \half \Re d(u_2 + iJu_2, u'_1 - iJu'_1)
\\
&= d(u_1, u'_2) + d(u_2, u'_1) = d(u_1, T_W u'_1) + d(T_W u_1, u'_1),
\end{align*}
which says that $T_W$ is skewsymmetric in the fermionic case. We may
denote by $\rSk(V)$ the real vector space of \textit{antilinear
skewsymmetric operators} on~$V$, so that $T_W \in \rSk(V)$.

The Cayley transformation formula~\eqref{eq:Cayley-transformation} 
also holds in Case~F, since $(1 + T_W)$ is invertible if it is 
defined, i.e., under the proviso that $\wb{W} \cap W_J = \{0\}$.

Many of these algebraic considerations remain valid when $V$ is
infinite dimensional, provided that $J - J_W$ is a Hilbert--Schmidt
operator; this is a version of the Shale--Stinespring condition
\cite{ShaleS65}. In that case, $B$ is a Fredholm operator of
index~$0$, with $\ker(\wb{W} \cap W_J)$ being a finite-dimensional
subspace. See \cite{Rhea, Araki87} for the full story.

The orthogonal group $\rO(V)$ acts transitively on the admissible 
polarizations. Thus we can put $W = g W_J$ and $J_W = g J g^{-1}$
as in Case~$B$. Now 
\begin{equation}
B = \half(1 - g J g^{-1} J) = g \. \half(g^t - J g^t J) = g p_g^t
\label{eq:defect-operator} % (2.15)
\end{equation}
so that $\ker B = \{0\}$ if and only if $p_g$ is invertible. 

The homogeneous space of compatible complex structures is stratified, 
with the strata labelled by
$k := \half \dim(\ker B) \in \{0,1,\dots,m\}$; the top number occurs
only when $g J g^{-1} = -J$, so that $W = W_{-J}$. The stratum with
$k = 0$ is an open dense subset (a ``big cell''), parametrized by
$T_g \in \rSk(V)$. Appendix~\ref{app:pin-repn} treats the cases where
$k > 0$.

Moreover, when $J_W = g T g^{-1}$, the relation
\eqref{eq:defect-operator} shows that the $\bC$-linear operators $p_g$
and $p_g^t$ on~$V_J$ satisfy
$\dim(\ker p_g) = \dim(\ker p_g^t) = \half \dim(\ker B)$.

The group-action formula \eqref{eq:group-action-on-T} continues to
hold in the fermionic case, but only as a \textit{local} action. If
$g \in \rSO_*(V)$ so that $p_g^{-1}$ exists, then $p_g + q_g S$ is
invertible for $S \in \rSk(V)$ in a neighbourhood of~$0$, and then
$g \. S$ also lies in~$\rSk(V)$. It is easy to check that
$gh\.S = g\.(h\.S)$ whenever $h\.S$ and $gh\.S$ are defined.

%%%%%%%%%%%%%%%%%%%%%%%%%%%%%%%%%%%%%%%%%%%%%%%%%%%%%

\section{Fock spaces} % 3
\label{sec:Fock-spaces}

The Fock spaces come in two types, bosonic and fermionic. The first
are function spaces; the second are (the underlying vector spaces of)
exterior algebras. Their construction is well known; we briefly
recapitulate it, in order to focus on their \textit{Gaussian vectors}.
These will be parametrized by $T \in \sD(V)$ in Case~B, and by
$T \in \rSk(V)$ in Case~F.

\subsection{The bosonic Fock space} % 3.1
\label{ssc:bosonic-Fock}

Case~B starts with a symplectic vector space $(V,s)$ of real
dimension~$2m$, equipped with a fixed compatible complex structure
$J$, as in Definition~\ref{df:complex-structure}(a). The bosonic
one-particle space $V$ is thus a complex Hilbert space of
dimension~$m$ with scalar product~\eqref{eq:scalar-product}.

\begin{defn} % 3.1
\label{df:bosonic-Fock}
The \textit{symmetric product} $u_1 \v u_2 \vyv u_k$ of vectors
$u_1,u_2,\dots,u_k \in V$ is obtained by symmetrizing their tensor
product $u_1 \ox u_2 \oxyox u_k$. These generate the bosonic
$k$-particle space $V^{\v k}$; their direct sum is the symmetric
algebra $S(V):= \bigoplus_{k=0}^\infty V^{\v k}$. Here
$V^{\v 0} := \bC\,\Om$ is just a one-dimensional space.

This is a prehilbert space under the scalar product determined by
\begin{equation}
\braket{u_1 \v u_2 \vyv u_k}{v_1 \v v_2 \vyv v_l}
:= \dl_{kl}\, \per[\braket{u_i}{v_j}],
\label{eq:B-multivector-pairing} % (3.1)
\end{equation}
where `$\per$' means the permanent of a square matrix; and by
declaring $\Om$ to be a unit vector in~$V^{\v 0}$. The completion
of~$S(V)$ to a Hilbert space is the \textit{bosonic Fock space}
$\sB(V)$.
\end{defn}

Since the symmetric product is commutative, it is possible, indeed 
advisable, to identify $\sB(V)$ with a space of functions. Firstly,
one can identify any vector $v \in V = V^{\v 1}$ with the
\textit{antilinear} mapping $u \mapsto 2^{-1/2} \braket{u}{v}$. This 
normalization, while unusual, turns out to be convenient. Next, a 
$k$-vector $v_1 \v v_2 \vyv v_k \in V^{\v k}$ can be identified with 
the homogeneous polynomial of degree~$k$,
$$
u \mapsto \frac{1}{2^{r/2}\,r!}\,
\braket{u \v u \vyv u}{v_1 \v v_2 \vyv v_k}.
$$

If $\{e_1,\dots,e_m\}$ is an orthonormal basis of~$V$, such 
polynomials coming from multivectors 
$e_{i_1} \v e_{i_2} \vyv e_{i_k}$ form an orthonormal basis
$\{\eps_\al\}$ for~$\sB(V)$, labelled by multi-indices~$\al$. In this
way $\sB(V)$ can be identified with the \textit{Segal--Bargmann space}
over~$\bC^m$, see \cite{Segal60, Bargmann61}. We follow the 
recommendation of Segal~\cite{Segal78, BaezSZ92} to view it as a 
space of \textit{antiholomorphic} (rather than
holomorphic~\cite{Bargmann61}) entire functions on~$V$.

\begin{defn} % 3.2
\label{df:Segal-Bargmann}
This complex-wave representation (Segal's term) of~$\sB(V)$ has the 
following structure:
$$
\sB(V):= \set{F\: V \to \bC \text{ antiholomorphic} : \|F\| < \infty},
$$
where $\|F\|$ denotes the $L^2$-norm with respect to the Gaussian
measure:
\begin{equation}
\|F\|^2 := \int_V |F(u)|^2 \,e^{-\half\braket{u}{u}} \,du 
\label{eq:BV-norm} % (3.2)
\end{equation}
and $du := (2\pi)^{-m} \,d\la(u)$ rescales the standard Lebesgue
measure $d\la(u)$ on $V \isom \bR^{2m}$.
\end{defn}

One checks that identifying $v \in V$ with the antilinear function
$F_v(u) := 2^{-1/2} \braket{u}{v}$ yields
$\|F_v\| = \braket{v}{v} = \|v\|$, so that the inclusion 
$V \hookto \sB(V)$ is an isometry. If $T \in \sD(V)$, the function
$(u,v) \mapsto \braket{u}{Tv}$ is antilinear in \textit{both}
variables, so the quadratic polynomial
\begin{equation}
H_T(u) := \half \braket{u}{Tu}
\label{eq:bosonic-hamiltonian} % (3.3)
\end{equation}
is antiholomorphic and defines a vector $H_T \in \sB(V)$.

\marker
It is important to recall~\cite{Bargmann61} that $\sB(V)$ is a
\textit{reproducing kernel Hilbert space}: one finds that
\begin{equation}
|F(u)|^2 \leq \|F\|^2 e^{\half \braket{u}{u}}
\word{for all}  F \in \sB(V).
\label{eq:Bargmann-inequality} % (3.4)
\end{equation}
A related feature of the bosonic Fock space is the availability of a
family of \textit{principal vectors} (in Bargmann's
terminology~\cite[Sect.~1c]{Bargmann61}; in the physics literature,
they are often called \textit{coherent states}). These are the
antiholomorphic entire functions $E_v(u) := e^{\half\braket{u}{v}}$,
for each $v \in V$. They belong to $\sB(V)$ since they have finite
norm:
\begin{equation}
\|E_v\|^2 
= \int_V \ovl{E_v(u)}\, E_v(u)\, e^{-\half\braket{u}{u}} \,du
= \int_V e^{\half\braket{v}{u} + \half\braket{u}{v}
- \half\braket{u}{u}} \,du = e^{\half\braket{v}{v}}.
\label{eq:principal-vectors} % (3.5)
\end{equation}
A similar calculation leads to the useful identity:
$$
\braket{E_v}{E_w} = e^{\half\braket{v}{w}}.
$$
While the principal vectors are not mutually orthogonal, they do form 
a total set in~$\sB(V)$, sometimes called an ``overcomplete basis''.

The principal vectors also exhibit the reproducing kernel, because
\begin{equation}
\braket{E_v}{F} = F(v)  \word{for all} v \in V.
\label{eq:reproducing-kernel} % (3.6)
\end{equation}
This is found by a straightforward integration by parts when 
$F = \eps_\al$ is a basis monomial, and the general case then follows
from Parseval's equality. 

This reproducing property shows that any bounded operator on $\sB(V)$ 
possesses an integral kernel. The following proposition was noticed by
Robinson and Rawnsley \cite[Thm.~1.7]{RobinsonR89}.

\begin{prop} % 3.3
\label{pr:operator-kernel}
If $A$ is a bounded operator in $\sB(V)$, there exists a
sesquiholomorphic function $K_A(u,v)$ on $V \x V$, such that
\begin{subequations}
\label{eq:operator-kernel} % (3.7)
\begin{equation}
AF(u) = \int K_A(u,v) F(v)\, e^{\half\braket{v}{v}} \,dv.
\label{eq:operator-kernel-action} % (3.7a)
\end{equation}
for any $F \in \sB(V)$. This integral kernel is given by
\begin{equation}
K_A(u,v) := \braket{E_u}{AE_v} = AE_v(u).
\label{eq:operator-kernel-formula} % (3.7b)
\end{equation}
\end{subequations}
\end{prop}

\begin{proof}
Denote by $A^\7$ the adjoint of the bounded operator~$A$, and let
$F \in \sB(V)$. Then
\begin{align*}
AF(u) &= \braket{E_u}{A F} = \braket{A^\7 E_u}{F}
= \int_V \ovl{A^\7 E_u(v)}\, F(v)\, e^{-\half\braket{v}{v}} \,dv
\\
&= \int_V \braket{A^\7 E_u}{E_v}\, F(v)\, e^{-\half\braket{v}{v}} \,dv
= \int_V \braket{E_u}{A E_v}\, F(v)\, e^{-\half\braket{v}{v}} \,dv.  
\end{align*}
On defining $K_A(u,v)$ by~\eqref{eq:operator-kernel-formula}, it is 
clear that \eqref{eq:operator-kernel-action} holds, and that 
$K_A(u,v)$ is antiholomorphic in~$u$. Moreover, since
$$
\ovl{K_A(u,v)} = \braket{AE_v}{E_u} = \braket{E_v}{A^\7E_u}
= K_{A^\7}(v,u), 
$$
it follows that $K_A(u,v)$ is holomorphic in~$v$.
\end{proof}

\subsection{The fermionic Fock space} % 3.2
\label{ssc:fermionic-Fock}

Case F begins with an orthogonal vector space $(V,d)$, again of real
dimension~$2m$, with a fixed compatible~$J$ obeying
Definition~\ref{df:complex-structure}(b). Again $V$
with~\eqref{eq:scalar-product} becomes an $m$-dimensional complex
Hilbert space.

\begin{defn} % 3.4
\label{df:fermionic-Fock}
The \textit{exterior product} $u_1 \w u_2 \wyw u_k$ of $k$ vectors
in~$V$ is obtained by skewsymmetrizing $u_1 \ox u_2 \oxyox u_k$. These
generate the fermionic $k$-particle space $V^{\w k}$; their direct sum
is the exterior algebra $\La(V) := \bigoplus_{k=0}^m V^{\w k}$. Again
we declare $V^{\w 0} := \bC\,\Om$ to be a one-dimensional space. This
$\La(V)$ is already a Hilbert space, of \textit{finite} dimension
$2^m$, under the scalar product determined by
\begin{equation}
\braket{u_1 \w u_2 \wyw u_k}{v_1 \w v_2 \wyw v_l}
:= \dl_{kl}\, \det[\braket{u_i}{v_j}],
\label{eq:F-multivector-pairing} % (3.8)
\end{equation}
and by declaring $\Om$ to be a unit vector in~$V^{\w 0}$. We baptize
$\La(V)$ with this scalar product as the \textbf{fermionic Fock
space} $\sF(V)$.
\end{defn}

Given an $\{e_1,e_2,\dots,e_m\}$ an \textit{oriented} orthonormal
basis of $V$, the exterior products 
$\eps_K:= e_{k_1} \w e_{k_2} \wyw e_{k_r}$ form an orthonormal basis
for $\sF(V)$, indexed by all $r$-element subsets
$K = \{k_1,k_2,\dots,k_r\} \subseteq \{1,2,\dots,m\}$ listed in
increasing order. We put $\eps_\emptyset:= \Om$ by convention.

Given $T \in \rSk(V)$, one can define a fermionic analogue of
\eqref{eq:bosonic-hamiltonian} by the finite sum
\begin{equation}
H_T:= \sum_{i,j=1}^m \braket{e_i}{Te_j}\, e_i \w e_j.
\label{eq:fermionic-hamiltonian} % (3.9)
\end{equation}
One checks that this $H_T \in \sF(V)$ does not depend on the chosen
orthonormal basis for~$V$, nor on its orientation, since
$\braket{e_j}{Te_i}\, e_j \w e_i = \braket{e_i}{Te_j}\, e_i \w e_j$
by skewsymmetry of~$T$.

\subsection{Gaussian vectors in Fock space}

The metaplectic and spin representations will presently be defined as 
projective representations of the symmetry groups on the Fock spaces
which permute a distinguished set of ``Gaussian'' vectors (up to
certain multiplicative factors). These are labelled by antilinear 
operators $T$ in~$\sD(V)$ or in~$\rSk(V)$, respectively; and in both 
cases are quadratic exponentials.

\begin{defn} % 3.5
\label{df:Gaussian-vector}
In each Fock space, the Gaussian $f_T$ is the exponential 
of $\half H_T$, defined as follows:
\begin{subequations}
\label{eq:Gaussian-vector} % (3.10)
\begin{enumerate}
\item % (a)
If $T \in \sD(V)$, then $f_T \in \sB(V)$ is defined by
\begin{equation}
f_T(u) := \exp(\half H_T(u)) 
= \exp\bigl( \quarter \braket{u}{Tu} \bigr).
\label{eq:Gaussian-vector-B} % (3.10a)
\end{equation}
Its norm in $\sB(V)$ equals $\|f_T\| = \det^{-1/4}(1 - T^2)$.
\item % (b)
If $T \in \rSk(V)$, then $f_T \in \sF(V)$ is defined by
\begin{equation}
f_T := \exp^\w(\half H_T)
:= \sum_{0 \leq 2r \leq m} \frac{1}{2^r\,r!}\, H_T^{\w r}.
\label{eq:Gaussian-vector-F} % (3.10b)
\end{equation}
Its norm in $\sF(V)$ equals $\|f_T\| = \det^{+1/4}(1 - T^2)$.
\qef
\end{enumerate}
\end{subequations}
\hideqef
\end{defn}

To obtain these formulas for the norms of Gaussians, we again separate
the two cases.

Most relevant norm computations on the bosonic Fock space $\sB(V)$
reduce to the calculation of certain Gaussian integrals. We restate
the basic lemma, see \cite[Sect.~1a]{Bargmann61} or Theorem~1 in
Appendix~A of~\cite{Folland89}, adapted to our normalized Lebesgue
measure $dx := (2\pi)^{-m} \,d\la(x)$.

\begin{lema} % 3.6
\label{lm:Gaussian-integral}
If $B \in M_m(\bC)$ is a complex symmetric matrix whose real part
is positive definite and if $c \in \bC^m$, then 
\begin{equation}
\int_{\bR^m} e^{-\half x\.Bx + c\.x} \,dx
= \frac{1}{\sqrt{\det B}}\, e^{\half c\.B^{-1}c},
\label{eq:Gaussian-integral} % (3.11)
\end{equation}
where $\sqrt{\det B}$ is defined by analytic continuation from the 
square root which is positive when $B$ is real.
\end{lema}

\begin{proof}
If $B$ is real (thus positive definite), the quadratic form
$x \mapsto x\.Bx$ can be diagonalized; since $\det B > 0$, take the
positive square root $\sqrt{\det B} > 0$. In the general case, write
$\det B = \bt_1\cdots\bt_m$ where the eigenvalues $\bt_i$ of~$B$ have
$\Re\bt_i > 0$; and choose
$\det^{-1/2} B := \bt_1^{-1/2} \cdots \bt_m^{-1/2}$ where each
$\bt_i^{-1/2}$ is the square root of~$\bt_i^{-1}$ with
$\Re \bt_i^{-1/2} > 0$.
\end{proof}

Now take $T \in \sD(V)$, recalling that $T^2$ is $\bC$-linear on $V$
with $1 - T^2 > 0$. Denote its eigenvalues by $1 - \la_k^2$, where
$0 \leq \la_k^2 < 1$, for $k = 1,\dots,m$; with matching eigenvectors
$\{e_1,\dots,e_m\}$ forming an orthonormal basis for~$V$. Then
$\{e_1,Je_1,\dots,e_m,Je_m\}$ is a real-linear basis for~$V$ and each
subspace $\linspan{e_k,Je_k}$ reduces~$T$. Since $T$ is antilinear and
symmetric, with eigenvalues $\{\pm \la_k\}$, we can select these $e_k$
so that $T e_k = \la_k\,Je_k$ and $T Je_k = \la_k\,e_k$. Accordingly,
$T(e_k \pm Je_k) = \pm\la_k(e_k \pm Je_k)$.

Remark that, as a $\bC$-linear operator on $V$,
$\det(1 - T^2) = \prod_{k=1}^m (1 - \la_k^2) > 0$.

Using this $T$-adapted basis, take $u \in V$ and put
$\al_k:= \braket{e_k}{u} \in \bC$, with
$\bar\al_k = \braket{u}{e_k}$ also. This implies that 
$$
\braket{Te_k}{u} = \braket{\la_k Je_k}{u} = -i\la_k \braket{e_k}{u}
= -i \la_k \al_k,
$$
and also $\braket{u}{Te_k} = +i \la_k \bar\al_k$. From the
antilinearity and symmetry of $T$ it follows that
\begin{align*}
\braket{u}{Tu} + \braket{Tu}{u} 
&= \sum_{k=1}^m \braket{u}{e_k} \braket{e_k}{Tu} 
+ \braket{Tu}{e_k} \braket{e_k}{u}
\\
&= \sum_{k=1}^m \braket{u}{e_k} \braket{u}{Te_k}
+ \braket{Te_k}{u} \braket{e_k}{u}
= \sum_{k=1}^m i\la_k(\bar\al_k^2 - \al_k^2) \in \bR.
\end{align*}

We can now compute the norm of the Gaussian $f_T$ defined 
by~\eqref{eq:Gaussian-vector-B}. Indeed,
\begin{align}
\|f_T\|^2 := \int_V |f_T(u)|^2 e^{-\half\braket{u}{u}} \,du
&= \int_V \exp \bigl( \quarter\braket{u}{Tu} + \quarter\braket{Tu}{u}
- \half\braket{u}{u} \bigr) \,du
\nonumber\\
&= \prod_{k=1}^m \int_\bC \exp \half \bigl\{ 
\ihalf \la_k(\bar\al_k^2 - \al_k^2) - \bar\al_k \al_k \bigr\} \,d\al_k
\nonumber \\
&= \prod_{k=1}^m (1 - \la_k^2)^{-1/2} = \det^{-1/2}(1 - T^2).
\label{eq:Gaussian-norm-B} % (3.12)
\end{align}
Here the formula \eqref{eq:Gaussian-integral} has been applied with
$c = 0$ and $B = \smash{\twobytwo{1}{-\la_k}{-\la_k}{1}}$, which is
positive definite by the hypothesis $T \in \sD(V)$. See also
\cite[Sect.~2]{Kamuk}.

The scalar product of two Gaussians is obtained from this formula 
\eqref{eq:Gaussian-norm-B} by polarization:
\begin{equation}
\braket{f_S}{f_T} = \det^{-1/2}(1 - TS)  \word{for}  S,T \in \sD(V).
\label{eq:Gaussian-pairing-B} % (3.13)
\end{equation}
This is less obvious than it looks, since one must check that the
linear operator $1 - TS$ is invertible and then determine the
(complex) square root of $\det^{-1}(1 - TS)$. We follow Robinson and
Rawnsley~\cite[Thm.~1.9]{RobinsonR89}. For $v \in V$ and 
$S,T \in \sD(V)$, there holds:
\begin{align*}
2 \Re\braket{v}{(1 - TS)v}
&= \braket{v}{(1 - TS)v} + \braket{(1 - TS)v}{v}
\\
&= \braket{v}{(1 - S^2)v} + \braket{(1 - T^2)v}{v}
+ \braket{v}{(S - T)Sv} - \braket{T(S - T)v}{v}
\\
&= \braket{v}{(1 - S^2)v} + \braket{(1 - T^2)v}{v}
+ \braket{Sv}{(S - T)v} - \braket{Tv}{(S - T)v}
\\
&= \braket{v}{(1-S^2)v} + \braket{v}{(1-T^2)v} + \|Sv - Tv\|^2 \geq 0,
\end{align*}
since the antilinear operators $(S - T)$ and $T$ are symmetric.
Therefore $(1 - TS)v = 0$ forces 
$\braket{v}{(1 - S^2)v} = \braket{v}{(1 - T^2)v} = 0$ and hence
$v = 0$; so $1 - TS$ is invertible. Also, 
$\Re\braket{v}{(1 - TS)v} > 0$ for $v \neq 0$. The invertible but not 
necessarily symmetric operators $B \in \End_\bC(V)$ for which 
$\Re \braket{v}{Bv} > 0$ for $v \neq 0$ form a contractible open 
neighbourhood of~$1$ in $\rGL_\bC(V)$ so there is a (unique)
continuous map $B \mapsto \det^{-1/2} B$ on this domain such that 
$(\det^{-1/2} B)^{-2} = \det B$ and $\det^{-1/2}(1) = +1$. 
And clearly, $\det^{-1/2} B > 0$ when $B$ is positive definite.

The matrix elements of $1 - TS$ (in any orthonormal basis
$\{u_1,\dots,u_m\}$ for~$V$) are
$$
\braket{u_j}{(1 - TS)u_k} 
= \braket{u_j}{u_k} - \braket{Su_k}{Tu_j}.
$$
The second term on the right depends linearly on~$T$ and antilinearly 
on~$S$; thus $\det(1 - TS)$, its reciprocal $\det^{-1}(1 - TS)$ and 
the chosen square root $\det^{-1/2}(1 - TS)$ all depend 
holomorphically on~$T$ and antiholomorphically on~$S$. Thus, the 
formula~\eqref{eq:Gaussian-pairing-B}, rewritten as
$$
\int_V \exp \bigl( \quarter\braket{u}{Tu} + \quarter\braket{Su}{u}
- \half\braket{u}{u} \bigr) \,du
= \det^{-1/2}(1 - TS),
$$
follows by analytic continuation from the diagonal case $T = S$,
which has already been established in~\eqref{eq:Gaussian-norm-B}.

One may generalize the preceding Gaussian integrals by adding 
first-degree terms to the exponents. Consider the expression
$$
I(T;v,w) := \int_V \exp \quarter \bigl\{ \braket{u}{Tu} 
+ \braket{Tu}{u} + 2\braket{u}{v} + 2\braket{w}{u} \bigr\} 
e^{-\half\braket{u}{u}} \,du.
$$
Using the $T$-adapted orthonormal basis $\{e_1,\dots,e_m\}$, and
taking $\al_k:= \braket{e_k}{u}$, $\bt_k:= \braket{e_k}{v}$,
$\ga_k:= \braket{e_k}{w}$, the preceding integral factors as
\begin{align*}
I(T;v,w) &= \prod_{k=1}^m \int_\bC \exp \quarter \bigl\{
i\la_k (\bar\al_k^2 - \al_k^2) + 2 \bar\al_k \bt_k 
+ 2 \al_k \bar\ga_k \bigr\} e^{-\half|\al_k|^2} \,d\al_k 
\\
&= \prod_{k=1}^m (1 - \la_k^2)^{-1/2}
\exp\bigl[ \quarter(1 - \la_k^2)^{-1} \{
- i\la_k \bt_k^2 + 2 \bt_k \bar\ga_k + i\la_k \bar\ga_k^2 \} \bigr]
\end{align*}
using \eqref{eq:Gaussian-integral}. The exponent on the right-hand
side is a sum of three matrix elements:
$$
\quarter \bigl\{ \braket{T(1 - T^2)^{-1} v}{e_k} \braket{e_k}{v}
 + 2 \braket{w}{e_k} \braket{e_k}{(1 - T^2)^{-1} v}
 + \braket{w}{e_k} \braket{e_k}{T(1 - T^2)^{-1} w} \bigr\}.
$$
Since $\sum_{k=1}^m \ketbra{e_k}{e_k} = 1$, the integral $I(T;v,w)$
becomes
\begin{subequations}
\label{eq:big-integral} % (3.14)
\begin{equation}
\det^{-1/2}(1 - T^2)\, \exp \quarter \bigl\{ 
\braket{T(1 - T^2)^{-1} v}{v}  + 2 \braket{w}{(1 - T^2)^{-1} v}
 + \braket{w}{T(1 - T^2)^{-1} w} \bigr\}.
\label{eq:big-integral-diagonal} % (3.14a)
\end{equation}

It is again possible to polarize this calculation to compute a similar
integral involving two distinct elements $S,T \in \sD(V)$. The result
must depend holomorphically on~$T$ and antiholomorphically on~$S$, and
must coincide with~\eqref{eq:big-integral-diagonal} when $S = T$. 
Those considerations yield the following formula.

\begin{lema} % 3.7
\label{lm:big-integral} 
The general Gaussian integral, depending on $S,T \in \sD(V)$, is
evaluated as
\begin{align}
& I(S,T;v,w) := \int_V \exp \quarter \bigl\{ \braket{u}{Tu} 
+ \braket{Su}{u} + 2\braket{u}{v} + 2\braket{w}{u} \bigr\} 
e^{-\half\braket{u}{u}} \,du
\label{eq:big-integral-general} % (3.14b)
\\
&= \det^{-1/2}(1 - TS)
\exp \quarter \bigl\{ 
\braket{S(1 - TS)^{-1} v}{v}  + 2 \braket{w}{(1 - TS)^{-1} v}
 + \braket{w}{(1 - TS)^{-1}T w} \bigr\}
\nonumber \\
&= \det^{-1/2}(1 - TS)
\exp \quarter \bigl\{ 
\braket{S(1 - TS)^{-1} v}{v}  + 2 \braket{w}{(1 - TS)^{-1} v}
 + \braket{w}{T(1 - ST)^{-1} w} \bigr\}.
\nonumber
\end{align}
where $I(T,T;v,w) \equiv I(T;v,w)$.
\qed
\end{lema}
\end{subequations}

The integral formula \eqref{eq:big-integral-general} is fairly well 
known, starting with Eq.~(1.18) in~\cite{Bargmann61} for
$V = \bC^m$, developed in \cite[Lemma~5]{Itzykson67} in terms of
matrices, exposited in Appendix~A of~\cite{Folland89} and again in
\cite[Thm.~1.10]{RobinsonR89}. Our version, 
Lemma~\ref{lm:big-integral}, is stated here mainly for its application
to the metaplectic representation in Section~\ref{sec:meta-spin}.

\marker
The treatment of fermionic Gaussians is analytically simpler, since
the fermionic Fock space is finite-dimensional. We summarize the
necessary exterior algebra, developed in detail in Sections 5.4
and~5.5 of~\cite{Polaris}. Starting with any $T \in \rSk(V)$, and any
oriented orthonormal basis $\{e_1,\dots,e_m\}$ of $V$, all
$b_{ij} := \braket{e_i}{Te_j}$ are entries of a skewsymmetric
$m \x m$ matrix~$B$, whose Pfaffian
$$
\Pf B := \frac{1}{2^m\,m!} \sum_{\sg\in S_n} 
(-1)^\sg \prod_{k=1}^m b_{\sg(2k-1),\sg(2k)}
$$
satisfies $(\Pf B)^2 = \det B$. (The sign of $\Pf B$ depends on the 
orientation.) 

For each ordered subset $L \subseteq \{1,\dots,m\}$, there is a
skewsymmetric submatrix $B_L$ with entries $\set{b_{ij} : i,j \in L}$;
denote its Pfaffian by $\Pf T_L \equiv \Pf B_L$. Then, by expanding
the exponential \eqref{eq:Gaussian-vector-F}, one finds that
\begin{align}
f_T &= \exp^\w(\half H_T)
= \sum_{0\leq 2r\leq m} \frac{1}{2^r\,r!}\, H_T^{\w r}
\nonumber \\
&= \sum_{0\leq 2r\leq m} \sum_{|K|=2r} \frac{1}{2^r r!}\, \eta_K\,
b_{k_1,k_2} \cdots b_{k_{2r-1,2r}} \,\eps_K
= \sum_{|K|\,\even} (\Pf T_K)\,\eps_K
\label{eq:Pfaffian-expansion} % (3.15)
\end{align}
where $\eta_K$ is the sign of the permutation putting the indices
$k_1,\dots,k_{2r}$ in increasing order. (Under a change of orientation
of the basis $\{e_i\}$, the signs of both $\Pf T_K$ and~$\eps_K$
remain the same or change together, so that the right-hand side
of~\eqref{eq:Pfaffian-expansion} is not affected.) Notice that the 
first term in the sum on the right is $\eps_\emptyset = \Om$.

The expansion \eqref{eq:Pfaffian-expansion} shows that the Gaussians
do not span the whole Fock space $\sF(V)$, but only the \textit{even
subspace} $\sF_0(V) \equiv \La^\even(V)$, since the monomials $\eps_K$
appearing in this expansion generate the even-degree subalgebra of
$\La^\8(V)$.

\begin{lema} % 3.9
\label{lm:Gaussian-norm-F}
If $T \in \rSk(V)$, the norm of the Gaussian element $f_T$ is given by
$$
\|f_T\|^2 = \sum_{|K|\,\even} |\Pf T_K|^2 = \det^{1/2}(1 - T^2).
$$
\end{lema}

\begin{proof}
The first equality is easy, since the multivectors $\eps_K$, with
$|K|$~even, form an orthonormal basis for $\sF_0(V)$, 
by~\eqref{eq:F-multivector-pairing}. We refer to Section~5.5 
of~\cite{Polaris} for the combinatorial proof that
\begin{equation}
\det(1 - TS) 
= \biggl( \sum_{|K|\,\even} \ovl{\Pf S_K}\,\Pf T_K \biggr)^2
\word{when}  S,T \in \rSk(V).
\label{eq:det-Pf-expansion} % (3.16)
\end{equation}
The complex conjugate of $\Pf S_K$ appears since each matrix entry
of~$TS$ is
$$
\braket{e_i}{TSe_j} = \braket{Se_j}{Te_i}
= \sum_{k=1}^m \braket{Se_j}{e_k}\,\braket{e_k}{Te_i}
= \sum_{k=1}^m \ovl{\braket{e_k}{Se_j}}\, \braket{e_k}{Te_i};
$$
so all summands in~\eqref{eq:det-Pf-expansion} depend holomorphically
on~$T$ and antiholomorphically on~$S$.

\goodbreak %%%

That, in turn, entails that one need only verify the diagonal case 
of~\eqref{eq:det-Pf-expansion}. Since $1 - T^2$ is symmetric and positive 
definite, by Lemma~\ref{lm:symmetry-of-T}, it follows that
$\det(1 - T^2) > 0$, and one may take the positive square root as the 
definition of $\det^{1/2}(1 - T^2)$.
\end{proof}

As in the bosonic case, we can polarize the formula for $\|f_T\|^2$ 
to obtain the scalar product of two Gaussian vectors. The result is:
\begin{equation}
\braket{f_S}{f_T} = \sum_{|K|\,\even} \ovl{\Pf S_K}\,\Pf T_K
= \det^{+1/2}(1 - TS)  \word{for}  S,T \in \rSk(V).
\label{eq:Gaussian-pairing-F} % (3.17)
\end{equation}
The first equality follows directly from the Pfaffian expansion
\eqref{eq:Pfaffian-expansion}. The second comes from
\eqref{eq:det-Pf-expansion}, with the same proviso as before about the
square root: one takes the positive square root when $S = T$, extended
by analytic continuation for $S \neq T$.

\section{The metaplectic and spin representations} % 4
\label{sec:meta-spin}

A common feature of the representations of the symmetry groups
$\rSp(V)$ and $\rO(V)$, that we shall soon construct on the respective
Fock spaces $\sB(V)$ and $\sF(V)$, is that their common unitary
subgroup
$$
\rU_J(V) = \rSp(V) \cap \rO(V)
$$
is easily implemented on free fields. These are, by definition,
operators $\phi(v)$ [or~$\psi(v)$] on Fock space, indexed by $v \in
V$, satisfying the canonical commutation [or~anticommutation]
relations. It is well known that, even when the configuration space
$V$ and the Fock space $\sF(V)$ are infinite dimensional, the
fermionic free fields are represented by bounded operators. However,
for bosonic fields that is never so, and such $\phi(v)$ will be
unbounded operators on~$\sB(V)$. For a detailed discussion of the
principles involved, consult Chapter~3 of~\cite{Emch72}.

\subsection{Weyl systems and bosonic free fields} % 4.1
\label{ssc:Weyl-system}

The issue of unbounded operators was finessed by Weyl, who introduced 
an exponential version of the canonical commutation relations, using 
unitary operators. Let $\sU(\sB(V))$ denote the group of unitary 
operators on~$\sB(V)$.

\begin{defn} % 4.1
\label{df:Weyl-system}
A \textit{Weyl system} on the symplectic vector space
$(V,s)$ is a strongly continuous map $\bt \: V \to \sU(\sB(V))$
that satisfies the relation
\begin{equation}
\bt(v) \bt(w) = e^{-\sihalf\,s(v,w)}\, \bt(v + w),
\word{for all} v,w \in V.
\label{eq:Weyl-relations} % (4.1)
\end{equation}
Thereby, $v \mapsto \bt(v)$ is a projective unitary representation of
the additive group of $V$ on~$\sB(V)$, with the $2$-cocycle $(v,w)
\mapsto e^{-\sihalf\,s(v,w)}$.
\end{defn}

The question of existence of Weyl systems is settled by the following 
\textit{example}:
\begin{subequations}
\label{eq:Weyl-system} % (4.2)
\begin{equation}
\bt(v) F : u \longmapsto e^{\quarter\braket{2u - v}{v}}\,F(u - v),
\word{for}  F \in \sB(V).
\label{eq:Weyl-system-one} % (4.2a)
\end{equation}

It is straightforward to verify that
\begin{equation}
\bt(v) E_w = e^{-\quarter\braket{v}{v + 2w}}\,E_{v+w}
\label{eq:Weyl-system-two} % (4.2b)
\end{equation}
\end{subequations}
for all $v,w \in V$; and consequently,
$$
\braket{\bt(v)E_{w'}}{\bt(v)E_w} = e^{\half\braket{w'}{w}}
= \braket{E_{w'}}{E_w}.
$$
Since the principal vectors $E_v$ form a total set in~$\sB(V)$, each
$\bt(v)$ is isometric and surjective, i.e., $\bt(v) \in \sU(\sB(V))$
as required.

\marker
Any unitary transformation $o \in \rU_J(V)$ of the one-particle 
space~$V$ lifts to a unitary operator $\Ga_\bos(o)$ on the Fock space
$\sB(V)$ in the obvious way:
$$
\Ga_\bos(o)F(u) := F(o^{-1}u).
$$
Using~\eqref{eq:Weyl-system}, one sees that these unitary 
operators intertwine the Weyl system:
\begin{equation}
\Ga_\bos(o)\,\bt(v)\,\Ga_\bos(o)^{-1} = \bt(ov)
\label{eq:Weyl-system-entwined} % (4.3)
\end{equation}
by applying both sides to any $F \in \sB(V)$.

For any (possibly unbounded) selfadjoint operator $A$ on $V$,
$t \mapsto \exp(itA)$ is a strongly continuous one-parameter subgroup
of $\rU_J(V)$, satisfying
\begin{align}
\Ga_\bos(\exp(itA)) E_v(u) &= E_v(\exp(-itA)u)
= \exp\bigl( \half\braket{\exp(-itA)u}{v} \bigr)
\nonumber \\
&= \exp\bigl( \half\braket{u}{\exp(itA)v} \bigr) = E_{\exp(itA)v}(u).
\label{eq:unitary-action} % (4.4)
\end{align}
By Stone's theorem, one may define a selfadjoint operator
$d\Ga_\bos(A)$ on~$\sB(V)$ by setting
$\exp(it\,d\Ga_\bos(A)) := \Ga_\bos(\exp(itA))$ for all $t \in \bR$.
Then $\sD_0(A) := \linspan\set{E_v : v \in \Dom A}$ is a dense
subspace of~$\sB(V)$, preserved by $d\Ga_\bos(A)$ since the
calculation \eqref{eq:unitary-action} shows that
$$
d\Ga_\bos(A)\,E_v = E_{Av}  \word{for all} v \in \Dom A.
$$
Therefore $\sD_0(A)$ is a core for $d\Ga_\bos(A)$ -- see, for instance,
\cite[Prop.~B.3]{Taylor86}.

It is immediate that
$$
\braket{E_w}{d\Ga_\bos(A)\,E_v} = \half \braket{w}{Av},
\word{for all} v,w \in \Dom A.
$$
If $w = \sum_{j=1}^k c_j v_j \in \Dom A$, then by putting
$F := \sum_{j=1}^k c_j E_{v_j} \in \sD_0(A)$, one obtains
$$
\braket{F}{d\Ga_\bos(A)F} 
= \frac{1}{2} \sum_{i,j=1}^k \bar c_i c_j \braket{v_i}{Av_j}
= \frac{1}{2} \braket{w}{Aw}.
$$
This shows that if $A$ is a \textit{positive} selfadjoint operator
on~$V$, then $d\Ga_\bos(A)$ is likewise a positive selfadjoint
operator on~$\sB(V)$: the unitary representation $\Ga_\bos$
of~$U_J(V)$ has \textit{positive energy}.

\marker
The \textit{bosonic free field} associated to the Weyl system
\eqref{eq:Weyl-system} is the family of unbounded selfadjoint
operators $\set{\phi(v) : v \in V}$ on~$\sB(V)$ defined implicitly by
$$
\exp(it\phi(v)) := \bt(\sqrt{2}\,tv)  \word{for} t \in \bR,
$$
whose existence is guaranteed by Stone's theorem. In other words,
\begin{equation}
\phi(v)F(u) := -\sqrt{2}\,i \,\ddto{t} \bt(tv) F(u)
\word{when} F \in \Dom \phi(v).
\label{eq:bosonic-field} % (4.5)
\end{equation}
The dense subspace $\sD_0 := \linspan\set{E_w : w \in V}$ of~$\sB(V)$
is a common core for all~$\phi(v)$; that follows from the next lemma.

\begin{lema} % 4.2
\label{lm:core-values}
For all $u,v,w \in V$, there holds:
\begin{equation}
\phi(v) E_w(u) = \tfrac{i}{\sqrt{2}}\,
\bigl( \braket{v}{w} - \braket{u}{v} \bigr) E_w(u).
\label{eq:core-values} % (4.6)
\end{equation}
\end{lema}

\begin{proof}
This is a direct computation, using \eqref{eq:Weyl-system-two}:
\begin{align*}
\phi(v) E_w(u) &= -\sqrt{2}\,i \,\ddto{t} \bt(tv) E_w(u)
= -\sqrt{2}\,i \,\ddto{t} \exp\bigl\{ -\quarter\braket{tv}{tv + 2w}
+ \half \braket{u}{tv + w} \bigr\}
\\
% &= -\sqrt{2}\,i \,\ddto{t} \exp\bigl\{ -\quarter t^2\braket{v}{v} 
% - \half t\braket{v}{w} + \half t\braket{u}{v}
% + \half\braket{u}{w} \bigr\}
% \\
&= \bigl( \tfrac{i}{\sqrt{2}}\, \braket{v}{w} 
- \tfrac{i}{\sqrt{2}}\, \braket{u}{v} \bigr)\, e^{\half\braket{u}{w}}
= \tfrac{i}{\sqrt{2}} \bigl( \braket{v}{w}
- \braket{u}{v} \bigr) E_w(u).
\tag*{\qed}
\end{align*}
\hideqed
\end{proof}

Note, in particular, that $\phi(v)F \in \sD_0$ whenever $F \in \sD_0$.
The commutation relations for the fields $\phi(v)$ follow directly.

\begin{prop} % 4.3
\label{pr:boson-fields}
For any $v,w \in V$, the following relation holds:
\begin{equation}
[\phi(v), \phi(w)] = 2i\,s(v,w),
\label{eq:boson-fields} % (4.7)
\end{equation}
where the right-hand side denotes a scalar operator on~$\sB(V)$.
\end{prop}

\begin{proof}
It is enough to evaluate the left-hand side on any $E_z$. Indeed,
\begin{align*}
\phi(v)\,\phi(w)\,E_z(u)
&= -\sqrt{2}\,i\,\ddto{t} \bt(tv)\,\phi(w)\,E_z(u)
= -\sqrt{2}\,i\,\ddto{t} e^{\quarter\braket{2u - tv}{tv}}
\phi(w)\, E_z(u - tv)
\\
&= \ddto{t} e^{\quarter\braket{2u - tv}{tv}} 
\bigl( \braket{w}{z} - \braket{u - tv}{w} \bigr)\, E_z(u - tv)
\\
&= \bigl( \braket{w}{z} - \braket{u}{w} \bigr)
\ddto{t} e^{\quarter\braket{2u - tv}{tv} + \half \braket{u - tv}{z}}
+ \braket{v}{w}\, E_z(u)
\\
&= \half \bigl\{
\bigl( \braket{w}{z} - \braket{u}{w} \bigr)
\bigl( \braket{u}{v} - \braket{v}{z} \bigr)
+ 2 \braket{v}{w} \bigr\} E_z(u).
\end{align*}
On exchanging $v \otto w$ and subtracting, several terms cancel, 
leaving only
$$
[\phi(v), \phi(w)]\,E_z 
= \bigl( \braket{v}{w} - \braket{w}{v} \bigr) E_z = 2i\,s(v,w)\,E_z
$$
which establishes \eqref{eq:boson-fields}.
\end{proof}

We make contact with the standard formalism by introducing the 
\textit{creation and annihilation operators} on the Fock space 
$\sB(V)$ as follows. For each $v \in V$, define
\begin{equation}
a_J^\7(v) := \half (\phi(v) - i\phi(Jv)) \word{and}
a_J(v) := \half (\phi(v) + i\phi(Jv)).
\label{eq:boson-pieces} % (4.8)
\end{equation}
It is immediate that:
$$
a_J^\7(Jv) = i\,a_J^\7(v) \word{and} a_J(Jv) = -i\,a_J(v).
$$ 
Therefore, the $a_J^\7(v)$ operators are $\bC$-linear in~$v$ and the
$a_J(v)$ are antilinear in~$v$. The boson field, reconstructed as
\begin{equation}
\phi(v) = a_J^\7(v) + a_J(v),
\label{eq:boson-splitting} % (4.9)
\end{equation}
is merely $\bR$-linear in~$V$.

On replacing $v \mapsto Jv$ in \eqref{eq:core-values}, it becomes
$$
\phi(Jv) E_w(u) = \tfrac{1}{\sqrt{2}}\,
\bigl( \braket{v}{w} + \braket{u}{v} \bigr) E_w(u).
$$
Combining both relations, we obtain
\begin{equation}
a_J^\7(v)E_w(u) = -\tfrac{i}{\sqrt{2}}\, \braket{u}{v}\, E_w(u)
\word{and}
a_J(v)E_w(u) = \tfrac{i}{\sqrt{2}}\, \braket{v}{w}\, E_w(u).
\label{eq:push-and-pull} % (4.10)
\end{equation}
Recalling the identification of $v \in V$ with the function
$u \mapsto \tfrac{1}{\sqrt{2}} \braket{u}{v}$, one sees that
$a_J^\7(v)$ acts as a multiplication operator (by~$-iv$) and $a_J(v)$
is a directional derivative. All annihilation operators kill the
vacuum $\Om = E_0$; and moreover,
$$
a_J^\7(v_1) a_J^\7(v_2) \dots a_J^\7(v_n)\,\Om
= (-i)^n\, v_1 \v v_2 \vyv v_n.
$$
Since the principal vectors $E_w$ form a total subset of~$\sB(V)$, 
the second relation in~\eqref{eq:push-and-pull} shows that the 
\textit{only} elements killed by all $a_J(v)$ are constant multiples 
of~$E_0$. Therefore, this \textit{vacuum sector} is the 
one-dimensional subspace~$\bC\,\Om$.

\subsection{Clifford algebras and fermionic free fields} % 4.2
\label{ssc:Clifford-mania}

The development of fermionic free fields over $V$ is somewhat simpler,
since they are implemented by \textit{bounded} operators $\psi(v)$ on
the Hilbert space $\sF(V)$ directly. (In the present context where
$\sF(V)$ is finite-dimensional, this is a trivial statement; but it
remains true more generally, from the work of Shale and
Stinespring~\cite{ShaleS65}.) Our approach here, already advertised
in~\cite{Rhea} and in \cite[Chaps.~5,\,6]{Polaris} differs somewhat
from the well-known formulation of~\cite{Araki87}, in which the role of
the complex structure $J$ on~$V$ is taken by a choice of a ``basis
projection''. However, both formulations are equivalent: a comparison
of the two approaches is found in~\cite{CalderonGRL18}.

The commutation relation \eqref{eq:boson-fields} must be replaced by
the anticommutation relation:
\begin{equation}
[\psi(u), \psi(v)]_+ = 2\,d(u,v)  \word{for all} u,v \in V,
\label{eq:fermion-fields} % (4.11)
\end{equation}
where $[A,B]_+ := AB + BA$ denotes an anticommutator.

These are the defining relations of the Clifford algebra $\Cl(V,d)$
over the orthogonal vector space $(V,d)$. This is the universal
$\bR$-algebra generated by elements $\set{c(v) : v \in V}$ where
$v \mapsto c(v)$ is $\bR$-linear and $[c(u), c(v)]_+ = 2\,d(u,v)$
holds. Its complexification%
\footnote{Up to isomorphism, $\bCl(V)$ is independent of the signature
of the symmetric bilinear form~$d$.}
is $\bCl(V) := \Cl(V,d) \ox_\bR \bC$ generated by the same $c(v)$,
acting by $\bC$-linear operators on the exterior algebra
$\La^\8 V_\bC$. By writing $c(u + iv) := c(u) + i\,c(v)$, the
anticommutation relations extend bilinearly to
$\set{c(w) : w \in V_\bC}$.

Since $\bCl(V) \isom M_{2^m}(\bC)$ is a simple matrix algebra for
$V_\bC \isom \bC^{2m}$, its irreducible representations are all
equivalent; indeed, they may be labelled by the orthogonal complex
structures, see \cite[Sect.~5.3]{Polaris}. One may take
$\psi(v) := \pi_J(c(v))$ for $v \in V$, where $\pi_J$ is the
irreducible representation on $\sF(V) = \sF(V_J)$. To express it
concretely, we write
\begin{equation}
\psi(v) = a_J^\7(v) + a_J(v), 
\label{eq:fermion-splitting} % (4.12)
\end{equation}
with
$$
a_J^\7(v) := \pi_J(c(P_+ v)),  \qquad  a_J(v) := \pi_J(c(P_- v)),
$$
where $P_\pm = \half(1 \mp iJ)$ are the polarization projectors
\eqref{eq:projectors}. These \textit{creation and annihilation
operators} on $\sF(V)$ are explicitly given by:
\begin{align}
a_J^\7(v) \bigl( v_1 \wyw v_k \bigr)
&:= v \w v_1 \w v_2 \wyw v_k, 
\nonumber \\
a_J(v) \bigl( v_1 \wyw v_k \bigr)
&:= \sum_{j=1}^k (-1)^{j-1} \braket{v}{v_j}\, 
v_1 \wyw \miss{v_j} \wyw v_k.      
\label{eq:birth-and-death} % (4.13)
\end{align}
For $k = 0$, this means $a_J^\7(v)\Om := v$ and $a_J(v)\Om := 0$.
Moreover, the second relation in~\eqref{eq:birth-and-death} shows that
the \textit{vacuum sector}, i.e., those elements killed by
all~$a_J(v)$, is the one-dimensional subspace~$\bC\,\Om$ in the
fermionic case, too.

From this definition it is clear that 
$a_J^\7(Jv) = i\,a_J^\7(v)$ and $a_J(Jv) = -i\,a_J(v)$, as in the 
bosonic case, so that $v \mapsto a_J^\7(v)$ is $\bC$-linear,
$v \mapsto a_J(v)$ is antilinear; and also $v \mapsto \psi(v)$ is
$\bR$-linear. The \textit{canonical anticommutation relations}:
\begin{equation}
[a_J^\7(u), a_J^\7(v)]_+ = [a_J(u), a_J(v)]_+ = 0 \word{and}
[a_J(u), a_J^\7(v)]_+ = \braket{u}{v}
\label{eq:CAR} % (4.14)
\end{equation}
follow at once by direct calculation. This also establishes
\eqref{eq:fermion-fields}, of course.

Once again, a unitary transformation $o \in \rU_J(V)$ lifts to a
unitary operator $\Ga_\fer(o)$ on the Fock space $\sF(V)$. The
prescription is:
$$
\Ga_\fer(o) \bigl( v_1 \wyw v_k \bigr) := ov_1 \wyw ov_k  
\word{for}  v_1,\dots,v_k \in V,
$$
together with $\Ga_\fer(o)\,\Om := \Om$. These unitary operators
intertwine the fermion fields
\begin{equation}
\Ga_\fer(o)\,\psi(v)\,\Ga_\fer(o)^{-1} = \psi(ov).
\label{eq:fermions-entwined} % (4.15)
\end{equation}
It suffices to check that the $\Ga_\fer(o)$ intertwine creation and
annihilation operators separately.

\subsection{Representations of the full symmetry groups} % 4.3
\label{ssc:full-group-repns}

The commutation and anticommutation relations for the fields
\eqref{eq:boson-fields} and \eqref{eq:fermion-fields} remain invariant
under the respective actions of the unitary group $U_J(V)$, since
\begin{align}
\Ga_\bos(o)\, [\phi(u), \phi(v)] \,\Ga_\bos(o)^{-1}
&= [\phi(ou), \phi(ov)] = 2i\,s(ou, ov) = 2i\,s(u,v);
\nonumber \\
\Ga_\fer(o)\, [\psi(u), \psi(v)]_+ \,\Ga_\fer(o)^{-1}
&= [\psi(ou), \psi(ou)]_+ = 2\,d(ou, ov) = 2\,d(u,v).
\label{eq:unitary-entwine} % (4.16)
\end{align}
But the bilinear forms $s$ and~$d$ each have a larger symmetry group:
$\rSp(V)$ and~$\rO(V)$ respectively. There ought to be other unitary
representations that extend the intertwining relations
\eqref{eq:unitary-entwine} to these larger groups. However, since
phase factors obviously cancel on the left-hand sides
of~\eqref{eq:unitary-entwine}, it is enough to find
\textit{projective} representations of these larger groups on the Fock
spaces.

There are, of course, well-known projective representations of these
symmetry groups, coming from true representations on certain
\textit{double covering groups}: the metaplectic and pin
representations of $\rSp(V)$ and $\rO(V)$, respectively. For standard
treatments of the metaplectic representation, we refer to
\cite{Folland89, Wallach18}; for the spin and pin representations,
see~\cite{GoodmanW09}. For the orthogonal group, we shall consider
first only the neutral component $\rSO(V)$ and its spin
representation: the full orthogonal group is treated in
Appendix~\ref{app:pin-repn}. Their covering groups $\rMp(V)$ and
$\rSpin(V)$ respectively, are central extensions of the basic symmetry
groups by the two-element group $C_2 = \{\pm 1\}$.

However, for infinite-dimensional $V$, as developed originally by 
Shale and Stinespring \cite{Shale62, ShaleS65}, these double covering 
groups are inadequate. There it is necessary to allow 
general phase factors; that is to say, one needs central extensions 
by the circle group $\rU(1)$:
$$
\begin{tikzcd}[row sep=small]
1 \rar & \rU(1) \rar & \rMpc(V) \rar & \rSp(V) \rar & 1,
\\
1 \rar & \rU(1) \rar & \rSpinc(V) \rar & \rSO(V) \rar & 1.
\end{tikzcd}
$$
Here we develop these extensions and their representations while 
keeping $V$ of finite (even) dimension.

\marker
To amplify the relations \eqref{eq:unitary-entwine} beyond $\rU_J(V)$ 
to the full symmetry groups, we need families of unitary operators on 
each Fock space, say $\nu(g) \in \sU(\sB(V))$ for $g \in \rSp(V)$; and
$\mu(g) \in \sU(\sF(V))$ for $g \in \rO(V)$, with the intertwining 
properties:
\begin{subequations}
\label{eq:full-entwine} % (4.17)
\begin{alignat}{2}
\nu(g)\,\phi(v) &= \phi(gv)\,\nu(g)
& \word{for all} & g \in \rSp(V), \ v \in V,
\label{eq:full-entwine-bos} % (4.17a)
\\[\jot]
\mu(g)\,\psi(v) &= \psi(gv)\,\mu(g)
& \word{for all} & g \in \rO(V), \ v \in V.
\label{eq:full-entwine-fer} % (4.17b)
\end{alignat}
\end{subequations}
It costs nothing to require also that $\nu(o) = \Ga_\bos(o)$ and
$\mu(o) = \Ga_\fer(o)$ for all $o \in \rU_J(V)$.

The new feature is that those $g$ outside of $\rU_J(V)$ do not commute
with the complex structure~$J$ on~$V$: they modify it to $g J g^{-1}$,
while our Fock spaces are built over the one-particle space $V = V_J$.
Hence, the vacuum sector (the one-dimensional subspace $\bC\,\Om$ of
Fock space), which depends implicitly on~$J$, is not preserved under
the extended group actions.

To see what happens to the vacuum sector, consider the effect of 
$J \mapsto g J g^{-1}$, for a fixed~$g$, on the creation and
annihilation operators. They undergo the following (well known)
\textit{Bogoliubov transformations}.

\begin{lema} % 4.4
\label{lm:Bogoliubov}
The creation and annihilation operators defined by either
\eqref{eq:boson-splitting} or \eqref{eq:fermion-splitting} transform
under the action of the full symmetry groups as follows:
\begin{equation}
a_{gJg^{-1}}^\7(gv) = a_J^\7(p_g v) + a_J(q_g v)  \word{and}
a_{gJg^{-1}}(gv) = a_J^\7(q_g v) + a_J(p_g v),
\label{eq:Bogoliubov} % (4.18)
\end{equation}
for all $v \in V$ and each $g \in \rSp(V)$ or $g \in \rO(V)$,
respectively.
\end{lema}

\begin{proof}
The fermionic operators obey the same defining
relations \eqref{eq:boson-pieces} as their bosonic counterparts:
$$
a_J^\7(v) = \half (\psi(v) - i\psi(Jv)) \word{and}
a_J(v) = \half (\psi(v) + i\psi(Jv)).
$$
This follows from linearity of $v \mapsto a_J^\7(v)$ and antilinearity
of $v \mapsto a_J(v)$, acting on~$\sF(V)$.

In the bosonic case, using $g = p_g + q_g$, it is clear that
\begin{align*}
a_{gJg^{-1}}^\7(gv) 
&= \half [\phi((p_g + q_g)v) - i\phi((p_g + q_g) Jv)]
\\
&= \half [\phi(p_g v) - i\phi(Jp_g v)]
+ \half [\phi(q_g v) + i\phi(Jq_g v)]
\\
&= a_J^\7(p_g v) + a_J(q_g v),
\end{align*}
and the second relation in \eqref{eq:Bogoliubov} follows by reversing 
the $\pm$~signs. The fermionic case comes from an identical 
calculation, with $\phi$ replaced by~$\psi$.
\end{proof}

On taking linear and antilinear parts of the field operators 
in~\eqref{eq:full-entwine}, one can re-express the intertwining 
properties of the desired unitary representations thus:
\begin{subequations}
\label{eqparts-entwine} % (4.19)
\begin{alignat}{2}
\nu(g)\, a_J^\7(v) &= a_{gJg^{-1}}^\7(gv)\, \nu(g),  \qquad
& \nu(g)\, a_J(v) &= a_{gJg^{-1}}(gv)\, \nu(g);
\label{eq:parts-entwine-bos} % (4.19a)
\\[\jot]
\mu(g)\, a_J^\7(v) &= a_{gJg^{-1}}^\7(gv)\, \mu(g),  \qquad
& \mu(g)\, a_J(v) &= a_{gJg^{-1}}(gv)\, \mu(g);
\label{eq:parts-entwine-fer} % (4.19b)
\end{alignat}
\end{subequations}
for all $v \in V$ and $g \in \rSp(V)$ or $g \in \rO(V)$, as the case 
may~be.

This shows that the transformed \textit{out-vacuum sectors} are the
respective one-dimensional subspaces $\bC\,\nu(g)\,\Om \leq \sB(V)$
and $\bC\,\mu(g)\,\Om \leq \sF(V)$. For each fixed symmetry~$g$, the
out-vacuum unit vectors $\nu(g)\,\Om$ and $\mu(g)\,\Om$ are thereby
determined, up to phase factors of modulus~$1$, by the conditions:
\begin{equation}
a_{gJg^{-1}}(gv)\, \nu(g)\,\Om = 0, \word{respectively}
a_{gJg^{-1}}(gv)\, \mu(g)\,\Om = 0, \quad\text{for all } v \in V.
\label{eq:out-vacua} % (4.20)
\end{equation}

\marker
Identification of the out-vacuum sectors requires separate arguments 
in the bosonic and fermionic cases, but the results end on common 
ground.

\begin{lema} % 4.5
\label{lm:out-vacuum-B}
If $g \in \rSp(V)$, the Gaussian vector $f_{T_g}$ lies in the 
out-vacuum sector of~$\sB(V)$:
\begin{equation}
a_{gJg^{-1}}(gv)\,f_{T_g} = 0  \word{for all} v \in V.
\label{eq:out-vacuum-B} % (4.21)
\end{equation}
\end{lema}

\begin{proof}
Using~\eqref{eq:Bogoliubov}, the condition~\eqref{eq:out-vacuum-B} 
can be written as
$$
a_J^\7(q_g v)\,f_{T_g} + a_J(p_g v)\,f_{T_g} = 0 
\word{for all} v \in V.
$$
On replacing $v$ by~$p_g^{-1} v$ and $T_g$ by any $T \in \sD(V)$,
this reduces to showing that
\begin{equation}
a_J^\7(Tv)\,f_T + a_J(v)\,f_T = 0  \word{for all} v \in V.
\label{eq:better-out-vacua} % (4.22)
\end{equation}
By \eqref{eq:Weyl-system-one} and the defining formula
\eqref{eq:Gaussian-vector-B} for~$f_T$,
\begin{align}
\phi(v) f_T(u) 
&= - \sqrt{2}\,i \ddto{t} \exp\bigl( 
\quarter \braket{2u - tv}{tv} \bigr) f_T(u - tv)
\nonumber \\
&= - \sqrt{2}\,i \ddto{t} \exp\bigl( \quarter \bigl(
\braket{2u - tv}{tv} + \braket{u - tv}{Tu - tTv} \bigr) \bigr)
\nonumber \\
&= - \frac{i}{2\sqrt{2}} \bigl(
\braket{2u}{v} - \braket{v}{Tu} - \braket{u}{Tv} \bigr) 
e^{\quarter \braket{u}{Tu}}
\nonumber \\
&= - \tfrac{i}{\sqrt{2}} 
\bigl( \braket{u}{v} - \braket{u}{Tv} \bigr) f_T(u).
\label{eq:push-one-Gaussian} % (4.23)
\end{align}
The last equality follows because $T$ is symmetric. 

Replacing $v$ by $Tv$ in \eqref{eq:push-one-Gaussian} yields 
$$
\phi(Tv) f_T(u) = - \tfrac{i}{\sqrt{2}} 
\bigl( \braket{u}{Tv} - \braket{Tv}{Tu} \bigr) f_T(u).
$$
Next, replacing $v$ by $Jv$ or $JTv$ 
in~\eqref{eq:push-one-Gaussian} and using
\eqref{eq:effect-of-J}, we obtain
\begin{align*}
\phi(Jv)\, f_T(u)
&= \tfrac{1}{\sqrt{2}}
\bigl( \braket{u}{v} + \braket{u}{Tv} \bigr) f_T(u),
\\
\phi(JTv)\, f_T(u)
&= \tfrac{1}{\sqrt{2}}
\bigl( \braket{u}{Tv} + \braket{Tv}{Tu} \bigr) f_T(u).
\end{align*}
Replacing these expressions in the left-hand side
of~\eqref{eq:better-out-vacua} yields complete cancellation:
$$
a_J^\7(Tv)\,f_T(u) + a_J(v)\,f_T(u)
= \half\bigl[ \phi(Tv) - i\phi(JTv) + \phi(v) + i\phi(Jv) \bigr]
f_T(u) = 0.
\eqno \qed
$$
\hideqed
\end{proof}

In the fermionic case, we obtain a similar formula
to~\eqref{eq:out-vacuum-B}, provided that $T_g$ is defined.

\begin{lema} % 4.6
\label{lm:out-vacuum-F}
If $g \in \rSO_*(V)$, the Gaussian vector $f_{T_g}$ lies in the 
out-vacuum sector of~$\sF(V)$:
\begin{equation}
a_{gJg^{-1}}(gv)\,f_{T_g} = 0  \word{for all} v \in V.
\label{eq:out-vacuum-F} % (4.24)
\end{equation}
\end{lema}

\begin{proof}
Just as in the proof of the previous lemma, one can rewrite
\eqref{eq:out-vacuum-F} as:
\begin{equation}
a_J^\7(q_g v)\,f_{T_g} + a_J(p_g v)\,f_{T_g} = 0 
\word{for all} v \in V.
\label{eq:out-vacuum-F-bis} % (4.25)
\end{equation}
In this case, the hypothesis that $g \in \rSO_*(V)$ means that
$p_g$~is invertible and $T_g$~is defined. Thus, one can replace $v$
by~$p_g^{-1} v$ and $T_g$ by any $T \in \rSk(V)$, thereby reducing the
condition \eqref{eq:out-vacuum-F} to a formal analogue
of~\eqref{eq:better-out-vacua}: we claim that if $T \in \rSk(V)$, then
$$
a_J^\7(Tv)\,f_T + a_J(v)\,f_T = 0  \word{for all} v \in V.
$$

This is now an $\bR$-linear relation in $\sF(V)$, so we can assume
that $\|v\| = 1$ and then choose an orthonormal basis
$\{e_1,\dots,e_m\}$ for $V$ with $e_1 = v$. The requirement is now to
show that
$$
a_J^\7(Te_1)\,f_T + a_J(e_1)\,f_T = 0 \word{whenever} T \in \rSk(V),
$$
where $f_T$ is given by \eqref{eq:Pfaffian-expansion}.

For completeness, we adapt the argument of \cite[Sec.~6.2]{Polaris}.
The Laplace expansion of the Pfaffian $\Pf T_K$ (with respect to the 
chosen basis) in the $k$th row and column yields
$$
\Pf T_K = \sum_{l\neq k}
(\pm)_{kl}\, \braket{e_k}{Te_l}\, \Pf T_{K\less\{k,l\}}\,,
$$
where $(\pm)_{kl}$ is the sign of the shuffle permutation 
$K \mapsto (k,l, K\less\{k,l\})$ and $K \subseteq \{1,\dots,m\}$ is
even. Applying $a(e_1)$ to $f_T = \sum_{|K|\,\even} (\Pf T_K)\,\eps_K$
gives
\begin{align*}
a_J(e_1) f_T 
&= a_J(e_1) \sum_{L\not\ni 1} \sum_{k\in L} \braket{e_1}{Te_k}
(\Pf T_{L\less\{k\}})\, e_1 \w e_k \wyw \eps_{L\less\{k\}}
\\
&= - \sum_{k=1}^m \sum_{M\not\ni 1} 
\braket{e_k}{Te_1}\, (\Pf T_M)\, e_k \w \eps_M
= - \sum_{M\not\ni 1} (\Pf T_M)\, Te_1 \w \eps_M,
\end{align*}
and since $a_J^\7(Te_1)\,f_T = Te_1 \w f_T$, this implies that
$$
a_J^\7(Te_1)\,f_T + a_J(e_1)\,f_T 
= \sum_{K\ni 1} (\Pf T_K)\, Te_1 \w \eps_K.
$$
On setting $L := K \less \{1\}$, the right-hand side can be further 
expanded as
\begin{align}
a_J^\7(Te_1)\,f_T + a_J(e_1)\,f_T 
&= \sum_{L\not\ni 1} (\Pf T_{\{1\}\cup L})\, Te_1 \w e_1 \w \eps_L
\nonumber \\
&= \sum_{L\not\ni 1} \sum_{k\in L} \braket{e_1}{Te_k}\,
(\Pf T_{L\less\{k\}})\, Te_1 \w e_1 \w e_k \w \eps_{L\less\{k\}}
\nonumber \\
&= \sum_{k=1}^m \braket{e_1}{Te_k} \sum_{|K|\,\even}
(\Pf T_K)\, e_k \w Te_1 \w e_1 \w \eps_K
\nonumber \\
&= \sum_{|K|\,\even} (\Pf T_K)\, Te_1 \w Te_1 \w e_1 \w \eps_K = 0,
\label{eq:kill-Gaussian-F} % (4.26)
\end{align}
by the cancellation $Te_1 \w Te_1 = 0$. The result follows.
\end{proof}

\subsection{The metaplectic representation} % 4.4
\label{ssc:metaplectic-repn}

The way is now clear to determine metaplectic representation of
$\rSp(V)$. Lemma~\ref{lm:out-vacuum-B} shows that $\nu(g)\Om$ is a
scalar multiple of~$f_{T_g}$. Unitarity of $\nu(g)$ demands that
\begin{equation}
\nu(g)\,\Om(u) := c_g f_{T_g}(u) 
= c_g \exp\bigl( \quarter\braket{u}{T_g u} \bigr),
\word{where} \|f_{T_g}\| = \det^{-1/4}(1 - T_g^2),
\label{eq:push-the-vacuum-B} % (4.27)
\end{equation}
thus the scalar prefactor $c_g \in \bC$ must satisfy
$|c_g| = \det^{+1/4}(1 - T_g^2)$.

\goodbreak %%%

Recall that $1 - T_g^2$ is positive definite, so $\det(1 - T_g^2) > 0$
and the powers $\det^{\pm 1/4}(1 - T_g^2)$ are unambiguous. The
crucial issue is the choice of the phase of~$c_g$. One possibility,
since $\det^{+1/4}(1 - T_g^2) = \det^{-1/4}(p_g p_g^t)$, is to take
$c_g := \det^{-1/2} p_g^t$; and this choice leads to the ``ordinary''
metaplectic representation \cite{Folland89, Wallach18} using the
double covering group $\rMp(V)$. However, this option cannot avail if
$V$ were infinite-dimensional, since $p_g$ will generally not have a
determinant at~all. To be able to extend these procedures to such~$V$
(subject to the Shale--Stinespring conditions \cite{Shale62,
ShaleS65}), we fix the phase by demanding that $c_g > 0$:
\begin{equation}
c_g := \det^{+1/4}(1 - T_g^2)  \word{for all}  g \in \rSp(V).
\label{eq:prefactor-B} % (4.28)
\end{equation}
 
Once the effect of $\nu(g)$ on the vacuum sector is fixed by
\eqref{eq:push-the-vacuum-B} and~\eqref{eq:prefactor-B}, its
intertwining properties \eqref{eq:parts-entwine-bos} allow to define
it on principal vectors $E_v = e^{\quarter\braket{v}{v}} \bt(v)\,\Om$.
This leads right away to an integral kernel for~$\nu(g)$, which was
originally derived by Vergne~\cite{Vergne77}. Explicitly
\cite{Hyperion, Kamuk}:
\begin{align*}
K_{\nu(g)}(u,v) &:= \nu(g) E_v(u)
= e^{\quarter \braket{v}{v}} \nu(g) \bt(v)\,\Om(u)
= e^{\quarter \braket{v}{v}}\, \bt(gv)\, f_{T_g}(u) 
\\
&= c_g \exp\bigl[ \quarter \bigl( \braket{v}{v}
+ \braket{2u - gv}{gv} + \braket{u - gv}{T_g(u - gv)} \bigr) \bigr].
\end{align*}
Now, using $g = (1 + T_g)p_g$, the right-hand side expands to
$$
c_g \exp\bigl[ \quarter \bigl( 
\braket{v}{v} - \braket{(1 + T_g)p_gv}{(1 - T_g^2)p_gv}
+ \braket{u}{T_g u} + 2\braket{u}{(1 - T_g) gv} \bigr) \bigr].
$$
By Lemma~\ref{lm:T-hat}, this simplifies to
\begin{align}
K_{\nu(g)}(u,v)
&= c_g \exp\bigl[ \quarter \bigl( \braket{u}{T_g u}
- \braket{T_g p_g v}{p_g^{-t}v} + 2\braket{u}{p_g^{-t}v} \bigr) \bigr]
\nonumber \\
&= c_g \exp\bigl[ \quarter \bigl( \braket{u}{T_g u}
+ 2 \braket{p_g^{-1} u}{v} + \braket{\That_g v}{v} \bigr) \bigr].
\label{eq:nu-kernel} % (4.29)
\end{align}

Now we can check that the intertwining property 
\eqref{eq:full-entwine-bos} does indeed hold.

\begin{lema} % 4.7
\label{lm:intertwine-one}
If $g \in \rSp(V)$ and $v \in V$, then
$\nu(g)\,\phi(v) = \phi(gv)\,\nu(g)$ as operators on~$\sB(V)$.
\end{lema}

\begin{proof}
It suffices to check that 
\begin{equation}
\nu(g)\,\bt(v) = \bt(gv)\,\nu(g),
\label{eq:intertwine-Weyl} % (4.30)
\end{equation}
since \eqref{eq:full-entwine-bos} will follow on replacing~$v$ by a
real multiple $tv$ and differentiating at $t = 0$.

The integral kernel for $\bt(v)$ is just
\begin{align*}
K_{\bt(v)}(u,w) &= \braket{E_u}{\bt(v)E_w}
= e^{-\quarter\braket{v}{v+2w}} \braket{E_u}{E_{v+w}}
\\
&= \exp \bigl[ \quarter \bigl( \braket{2u}{v + w}
- \braket{v}{v + 2w} \bigr) \bigr].
\end{align*}

Taking $A := \nu(g)\,\phi(v)$ and $B := \phi(gv)\,\nu(g)$, we then 
compute, using \eqref{eq:big-integral-general}:
\begin{align*}
K_A(u,w) &= \int_V K_{\nu(g)}(u,t)\, K_{\bt(v)}(t,w)\, 
e^{-\half\braket{t}{t}} \,dt
\\
&= c_g\, e^{\quarter(\braket{u}{T_g u} - \braket{v}{v+2w})}\,
I(\That_g, 0; v + w, p_g^{-1}u)
\\
&= c_g\, e^{\quarter(\braket{u}{T_g u} - \braket{v}{v+2w})}\,
\exp \bigl[ \quarter \bigl( \braket{\That_g(v + w)}{v + w}
+ 2\braket{p_g^{-1}u}{v + w} \bigr) \bigr]
\\
&= K_{\nu(g)}(u,w)\, \exp \bigl[ \quarter \bigl(
2\braket{p_g^{-1}u}{v} - \braket{(1 - \That_g)v}{v} 
- 2\braket{(1 - \That_g)v}{w} \bigr) \bigr].
\end{align*}

In like manner,
\begin{align*}
K_B(u,w) &= \int_V K_{\bt(gv)}(u,t)\, K_{\nu(g)}(t,w)\, 
e^{-\half\braket{t}{t}} \,dt
\\
&= c_g\, e^{\quarter(\braket{2u - gv}{gv} - \braket{\That_g w}{w})}\,
I(0, T_g; p_g^{-t}w, u - gv)
\\
&= c_g\, e^{\quarter(\braket{2u - gv}{gv} - \braket{\That_g w}{w})}\,
\exp \bigl[ \quarter \bigl( 2\braket{p_g^{-1}(u- gv)}{w}
+ \braket{u - gv}{T_g(u - gv)} \bigr) \bigr]
\\
&= K_{\nu(g)}(u,w)\, \exp \bigl[ \quarter \bigl(
2\braket{u}{(1 - T_g)gv} - \braket{gv}{(1 - T_g)gv}
- 2\braket{p_g^{-1}gv}{w} \bigr) \bigr].
\end{align*}
Using the relations $(1 - T_g)g = (1 - T_g^2)p_g = p_g^{-t}$, that 
follows from~\eqref{eq:fancy-formula}; and $1 - \That_g = p_g^{-1}g$,
from Lemma~\ref{lm:T-hat}, we see that the exponents on both 
right-hand sides coincide, and thus $K_A(u,w) \equiv K_B(u,w)$, as
required.
\end{proof}

\marker
Although the formula \eqref{eq:nu-kernel} fully describes the
metaplectic representation, because the principal vectors generate a
dense subspace of~$\sB(V)$, it is for that very reason not
well-adapted to a comparison with the spin representation of the
fermionic case. It is better to consider the effect of $\nu(g)$ on
Gaussian elements of~$\sB(V)$.

\begin{prop} % 4.8
\label{pr:push-Gaussians} 
For all $g,h \in \rSp(V)$, the following identity is valid:
\begin{equation}
\nu(g)\,f_{T_h}
= \det^{+1/4}(1 - T_g^2)\, \det^{-1/2}(1 - T_h\That_g)\, f_{T_{gh}}\,.
\label{eq:push-Gaussian-B} % (4.31)  
\end{equation}
\end{prop}

\begin{proof}
Using Lemma~\ref{lm:big-integral}, we find that
\begin{align*}
\nu(g)\,f_{T_h}(u)
&= \int_V K_{\nu(g)}(u,v) f_{T_h}(v) \,e^{-\half\braket{v}{v}} \,dv
\nonumber \\
&= c_g \int_V \exp\bigl[ \quarter \bigl( \braket{u}{T_g u}
+ \braket{v}{T_h v} + 2\braket{p_g^{-1}u}{v} 
+ \braket{\That_g v}{v} \bigr) \bigr] e^{-\half\braket{v}{v}} \,dv
\nonumber \\
&= c_g\, e^{\quarter\braket{u}{T_g u}} I(\That_g, T_h; 0, p_g^{-1}u)
\nonumber \\
&= c_g \det^{-1/2}(1 - T_h\That_g)\, 
\exp\bigl[ \quarter \bigl\{ \braket{u}{T_g u}
+ \braket{p_g^{-1}u}{T_h(1 - \That_g T_h)^{-1} p_g^{-1} u} \bigr\}
\bigr]
\end{align*}
and using the formula \eqref{eq:T-compositions}, the third factor
simplifies to 
$\exp\bigl[ \quarter \braket{u}{T_{gh} u} \bigr] = f_{T_{gh}}(u)$.
\end{proof}

Among other things, the last result shows that 
$\nu(g)\,\nu(h)\,\Om$ and $\nu(gh)\,\Om$ are both multiples of 
$f_{T_{gh}}$; they differ by a multiplicative factor $c(g,h)$.
This factor may be read off from \eqref{eq:push-the-vacuum-B}
and~\eqref{eq:push-Gaussian-B}, but it is instructive to compute it 
directly from the integral kernels~\eqref{eq:nu-kernel}.

\begin{prop} % 4.9
\label{pr:nu-composed}
The metaplectic operators $\nu$ form a projective representation of
$\rSp(V)$:
\begin{subequations}
\label{eq:cocycle-B} % (4.32)
\begin{equation}
\nu(g) \nu(h) = c(g,h)\,\nu(gh),
\label{eq:cocycle-B-defn} % (4.32a)
\end{equation}
where the group $2$-cocycle $c(g,h) \in \bT$ is given explicitly by
\begin{equation}
c(g,h) := \exp\bigl[ i\Arg\det^{-1/2}(1 - T_h\That_g) \bigr].
\label{eq:cocycle-B-formula} % (4.32b)
\end{equation}
\end{subequations}
\end{prop}

\begin{proof}
The integral kernel of $\nu(g) \nu(h)$ is computed directly as
\begin{align*}
& \int_V K_{\nu(g)}(u,w) K_{\nu(h)}(w,v) 
e^{-\half \braket{w}{w}} \,dw
\\
&= c_g c_h e^{\quarter(\braket{u}{T_g u} + \braket{\That_h v}{v})}\,
I(\That_g, T_h; p_h^{-t} v, p_g^{-1} u)
\\
&= c_g c_h \det^{-1/2}(1 - T_h\That_g) 
\exp \bigl[ \quarter\bigl( \braket{u}{T_{gh}u}
+ 2\braket{p_{gh}^{-1}u}{v} + \braket{\That_{gh}v}{v} \bigr) \bigr],
\end{align*}
where the formulas \eqref{eq:compositions} have again been used.
The right-hand side equals $c(g,h)\, K_{\nu(gh)}(u,v)$, where
\begin{align}
c(g,h) &= c_g c_h c_{gh}^{-1} \det^{-1/2}(1 - T_h\That_g)  
\nonumber \\
&= \det^{1/4}(1 - T_g^2) \det^{1/4}(1 - T_h^2)
\det^{-1/4}(1 - T_{gh}^2) \det^{-1/2}(1 - T_h\That_g) 
\nonumber \\
&= \exp\bigl[ i\Arg\det^{-1/2}(1 - T_h\That_g) \bigr],
\label{eq:explicit-cocycle-B} % (4.33)
\end{align}
as claimed. In the exponent, `$\Arg$' denotes the principal value of
the argument, determined by analytic continuation from the positive
reals. Notice that since $1 - T_g^2 = (p_g p_g^t)^{-1}$, the cocycle 
can also be written as $c(g,h) 
= \exp\bigl[ -i\Arg\det^{-1/2}(p_g^{-1} p_{gh} p_h^{-1}) \bigr]$.
\end{proof}

Now the Gaussians $f_T$, being \textit{even} functions of~$u$, densely
span only the even subspace
$$
\sB_0(V) := \set{f \in \sB(V) : f(-u) \equiv f(u)}.
$$
One can rewrite \eqref{eq:push-Gaussian-B} as 
\begin{equation}
\nu(g)\,f_S = c_g\, \vf_g(S)\, f_{g\.S}  \word{where} 
\vf_g(S) := \det^{-1/2}(1 - S\That_g),
\label{eq:act-on-Gaussians-B}  % (4.34) 
\end{equation}
for all $S \in \sD(V)$, with $g \. S$ denoting the action of $\rSp(V)$
on~$\sD(V)$, given by \eqref{eq:group-action-on-T} and satisfying
$g \. T_h = T_{gh}$. Since $g \. 0 = T_g$, it follows that $\Om$ is a
cyclic vector for the action of~$\nu$ on~$\sB_0(V)$.

The complementary subspace $\sB_1(V)$ of odd functions is spanned by
vectors of the form $\phi(v)\,f_S$, on account
of~\eqref{eq:push-one-Gaussian}: 
$$
\phi(v) f_S(u) = - \tfrac{i}{\sqrt{2}} \braket{u}{(1 - S)v}\, f_S(u);
\word{i.e.,} \phi(v) f_S = -i(1 - S)v\, f_S. 
$$
Notice that $(1 - S)$ is invertible, and any odd function in~$\sB(V)$
is of the form $w\,h$ with $w \in V$ and $h \in \sB_0(V)$.
Now, using Lemma~\ref{lm:intertwine-one}, the relation
$$
\nu(g)\,[\phi(v)\,f_S] = \phi(gv)\,\nu(g)\,f_S 
\propto \phi(gv)\,f_{g\.S}
$$
shows that $\nu$ also acts irreducibly on~$\sB_1(V)$. 

These subrepresentations of $\nu$ on $\sB_0(V)$ and $\sB_1(V)$ are
inequivalent; this can be seen by restricting $\nu$ to the unitary
subgroup $\rU_J(V)$: the action of $\Ga_\bos(\rU_J(V))$ on $\sB_0(V)$ 
has a fixed point, namely~$\Om$; whereas it possesses no fixed vector 
in~$\sB_1(V)$.

\subsection{The spin representation} % 4.5
\label{ssc:spin-repn}

The spin representation $\mu$ of the orthogonal group $\rO(V)$, acting 
on the Fock space $\sF(V)$, requires some more preparation, since the 
out-vacuum sector $\bC \mu(g)\,\Om$ has only been identified for 
elements~$g$ in the subset $\rSO_*(V)$. To characterize it in the 
general case, we borrow the procedure of~\cite{Rhea} which follows 
the work of Ruijsenaars~\cite{Ruijsenaars78} on Bogoliubov 
transformations.

Recall from the discussion after~\eqref{eq:defect-operator} that the 
manifold $\rO(V)$ is stratified by
$$
\dim(\ker p_g) = \dim(\ker p_g^t) =: k \in \{0,1,\dots,m\}.
$$
The top stratum, where $\ker p_g = \{0\}$, is just $\rSO_*(V)$. When
$k > 0$, the linear operator $p_g$ maps $(\ker p_g)^\perp$ to
$\im p_g = (\ker p_g^t)^\perp$ bijectively. Choose an orthonormal
basis $\{e_1,\dots,e_k\}$ for $\ker p_g^t$; and for $i = 1,\dots,k$,
define $r_i \in \End_\bR V$ by
$$
r_i(e_i) := Je_i\,;  \quad r_i(Je_i) := e_i\,; \word{and}
r_i(v) := -v \word{when} d(v, e_i) = d(v, Je_i) = 0.
$$
Each $-r_i$ is the reflection in the real hyperplane of $(V,d)$
orthogonal to $e_i + Je_i$.

Now put $r := r_1 \cdots r_k \in \rO(V)$; then $r$~acts as~$-1$ on
$\im p_g$ and maps the complementary subspace $\ker p_g^t$ to itself.
Moreover, $q_g(\ker p_g) \subseteq \ker p_g^t$ since
$p_g^t q_g = - q_g^t p_g$ by~\eqref{eq:p-q-relations}. Thus
$h := rg \in \rO(V)$ decomposes as the direct sum of two maps whose 
linear part is the invertible map $p_h = rq_g \oplus p_g$ from
$\ker p_g \oplus (\ker p_g)^\perp$ onto
$\ker p_g^t \oplus (\ker p_g^t)^\perp$. It follows that
$h \in \rSO_*(V)$.

The next proposition gives the out-vacuum sector for~$g$.

\begin{prop} % 4.10
\label{pr:defect-vacuum}
If $\dim(\ker p_g) = k$ and $r := r_1 \cdots r_k$ as above, then
$h := rg \in \rSO_*(V)$, and for each $v \in V$ there holds:
\begin{equation}
\bigl( a_J^\7(q_g v) + a_J(p_g v) \bigr) 
\bigl[ e_1 \wyw e_k \w f_{T_h} \bigr] = 0.
\label{eq:defect-vacuum} % (4.35)
\end{equation}
\end{prop}

\begin{proof}
When $g \in \rSO_*(V)$, so that $k = 0$, the relation
\eqref{eq:defect-vacuum} reduces to \eqref{eq:out-vacuum-F-bis},
already established. For the general case with $k > 0$, we recall the
argument of \cite[Sect.~6.3]{Polaris}. We abbreviate
$\ze := e_1 \wyw e_k \in \La^k(\ker p_g^t)$.
From~\eqref{eq:birth-and-death} it follows that
\begin{align*}
a_J^\7(q_g v)[\ze \w f_{T_h}] &= a_J^\7(q_g v)[\ze] \w f_{T_h}
= (-1)^k \ze \w a_J^\7(q_g v)[f_{T_h}],
\\
a_J(p_g v)[\ze \w f_{T_h}]
&= a_J(p_g v)[\ze] \w f_{T_h} + (-1)^k \ze \w a_J(p_g v)[f_{T_h}],
\end{align*}
since $a_J(p_g v)$ is an antiderivation. But $a_J(p_g v)[\ze] = 0$
since $p_g v \in \im p_g = (\ker p_g^t)^\perp$.

It is enough to consider the two cases $v \in \ker p_g$ and
$v \in (\ker p_g)^\perp$. When $v \in \ker p_g$, so that
$q_g v \in \ker p_g^t$, then $a_J(p_g v) = 0$ and
$a_J^\7(q_g v)[\ze] = 0$, and thus \eqref{eq:defect-vacuum} holds.

On the other hand, if $v \in (\ker p_g)^\perp$, then 
$p_h v = p_g v$ and $q_h v = q_g v$, so the left-hand side of
\eqref{eq:defect-vacuum} reduces to $(-1)^k \ze \w \eta$, where
$\eta := \bigl( a_J^\7(q_h v) + a_J(p_h v) \bigr)[f_{T_h}]$.
But in that case, \eqref{eq:out-vacuum-F-bis} shows that $\eta = 0$,
since $h \in \rSO_*(V)$. Therefore, \eqref{eq:defect-vacuum} holds
generally.
\end{proof}

Use of the negative-reflections $r_i$ allows one to extend $\mu$ to 
the full orthogonal group $\rO(V)$. The details are given in  
Appendix~\ref{app:pin-repn}.

\marker
The fermionic Fock space $\sF(V)$, like its bosonic counterpart,
is $\bZ_2$-graded, by taking $\sF(V) = \sF_0(V) \oplus \sF_1(V)$ where
$\sF_0(V) := \bigoplus_{k\,\even} V^{\w k}$, a subspace of $\La(V)$ of
dimension~$2^{m-1}$. It is clear from the expansion
\eqref{eq:Pfaffian-expansion} that all Gaussians $f_T$ lie
in~$\sF_0(V)$ and indeed span that subspace. The multivectors
$e_1 \wyw e_k \w f_{T_h}$ lie in $\sF_0(V)$ or~$\sF_1(V)$ according
as $k$~is even or~odd.

Analogously to~\eqref{eq:push-the-vacuum-B}, it follows that
\begin{equation}
\mu(g)\,\Om = c_g\, e_1 \wyw e_k \w f_{T_h}
\label{eq:push-the-vacuum-F} % (4.36)
\end{equation}
where now 
$$
\bigl\| e_1 \wyw e_k \w f_{T_h} \bigr\| = \|f_{T_h}\|
= \det^{+1/4}(1 - T_h^2),
$$
and the scalar prefactor $c_g \in \bC$ must now satisfy
$|c_g| = \det^{-1/4}(1 - T_h^2)$.

We may again fix the phase factor by imposing $c_g > 0$, taking
$c_g := \det^{-1/4}(1 - T_h^2)$ for all $g \in \rSO(V)$. Note also 
that 
$$ 
\det^{-1/4}(1 - T_h^2) = \det^{+1/4}(p_h p_h^t)
= \det^{+1/4}(p_h^t p_h) 
= \det^{+1/4}\bigl( p_g^t p_g|_{(\ker p_g)^\perp} \bigr)
$$
whenever $g \in \rSO(V)$, where $k$~is even; this expression is 
independent of the choice of~$r$.

For simplicity, and for a clearer comparison with the metaplectic
case, we focus on the top stratum where $k = 0$, wherein
\begin{equation}
c_g := \det^{-1/4}(1 - T_g^2)  \word{for all}  g \in \rSO_*(V).
\label{eq:prefactor-F} % (4.37)
\end{equation}

\marker
At this stage, one can define operators $\mu(g)$ on~$\sF_0(V)$, for
$g \in \rSO_*(V)$, by analogy with \eqref{eq:act-on-Gaussians-B}.
Since $p_g$ is invertible, $p_g + q_gS$ is also invertible for $S$ in
a neighbourhood of~$0$ in $\rSk(V)$, so that $g\.S$ exists, by
\eqref{eq:group-action-on-T}. The Gaussians $f_S$ for which
$g\.S$ exists span the vector space $\sF_0(V)$, and we can define
an operator $\mu(g)$ on~$\sF_0(V)$ by
\begin{equation}
\mu(g) f_S := c_g \vf_g(S) f_{g\.S} 
\label{eq:act-on-gaussians-F} % (4.38)
\end{equation}
where $\vf_g(S) \in \bC$ is a factor to be determined; it must be 
chosen so that $\mu(g)$ is unitary.

\begin{lema} % 4.11
\label{lm:scalar-factor-F}
For $g \in \rSO_*(V)$, taking $c_g$ as in~\eqref{eq:prefactor-F}, the 
prescription \eqref{eq:act-on-gaussians-F} will define a unitary 
operator on $\sF_0(V)$ provided that
\begin{equation}
\vf_g(S) := \det^{+1/2}(1 - S\That_g).
\label{eq:scalar-factor-F} % (4.39)
\end{equation}
\end{lema}

\begin{proof}
Applying \eqref{eq:act-on-gaussians-F} to any $S,T \in \rSk(V)$ such
that $g\.S$ and $g\.T$ are defined, \eqref{eq:Gaussian-pairing-F}
yields
\begin{align*}
\braket{\mu(g)\,f_S}{\mu(g)\,f_T} 
&= \braket{c_g \phi_g(S) f_{g\.S}}{c_g \phi_g(T) f_{g\.T}}
\\
&= c_g^2 \ovl{\phi_g(S)} \phi_g(T)\, \braket{f_{g\.S}}{f_{g\.T}}
\\
&= \ovl{\phi_g(S)} \phi_g(T)
\det^{-1/2}(1 - T_g^2) \det^{1/2}(1 - (g\.T)(g\.S)).
\end{align*}
To simplify the last two factors on the right-hand side, one may use
\eqref{eq:group-action-on-T}, recalling from~\eqref{eq:fancy-formula}
that $1 - T_g^2 = (p_g p_g^t)^{-1}$ and also
$1 - \That_g^2 = (p_g^t p_g)^{-1}$, to get
\begin{align*}
\MoveEqLeft{p_g^t (1 - (g\.T)(g\.S)) p_g}
\\
&= p_g^t (1 - T_g^2) p_g - T (1 - \That_g T)^{-1} p_g^{-1} T_g p_g
- p_g^t T_g p_g^{-t} S (1 - \That_g S)^{-1}
\\
&\qquad 
- T (1 - \That_g T)^{-1} p_g^{-1} p_g^{-t} S (1 - \That_g S)^{-1}
\\
&= 1 + (1 - T \That_g)^{-1} T \That_g + \That_g S(1 - \That_g S)^{-1}
- (1 - T \That_g)^{-1} T (1 - \That_g^2) S (1 - \That_gS)^{-1}
\\
&= (1 - T \That_g)^{-1} \bigl\{ (1 - T \That_g)(1 - \That_gS)
+ T \That_g (1 - \That_gS) 
\\
&\hspace*{7em} + (1 - T \That_g) \That_g S
- T (1 - \That_g^2) S \bigr\} (1 - \That_gS)^{-1}
\\
&= (1 - T \That_g)^{-1} (1 - TS) (1 - \That_gS)^{-1}.
\end{align*}
Therefore,
\begin{align*}
\braket{\mu(g)\,f_S}{\mu(g)\,f_T} 
&= \ovl{\phi_g(S)} \det^{-1/2}(1 - \That_g S)\,
\phi_g(T) \det^{-1/2}(1 - T \That_g)\, \det^{1/2}(1 - TS)
\\
&= \det^{1/2}(1 - TS) = \braket{f_S}{f_T},
\end{align*}
provided $\phi_g(S)$ and~$\phi_g(T)$ are given
by~\eqref{eq:scalar-factor-F}, thereby cancelling the $\det^{-1/2}$
factors. Since the Gaussians $f_S$, for $S$ in any neighbourhood
of~$0$, form a total set in $\sF_0(V)$, it follows that each operator
$\mu(g)$ is unitary on the even Fock subspace.
\end{proof}

\begin{lema} % 4.12
\label{lm:two-scalar-factors}
Assume that $g,h,gh \in \rSO'_*(V)$. Take $S \in \rSk(V)$ such that
$h\.S$ and $gh\.S$ are defined. Then
\begin{equation}
\vf_{gh}(S) = \det^{-1/2}(1 - T_h\That_g)\, \vf_h(S)\, \vf_g(h\.S).
\label{eq:two-scalar-factors} % (4.40)
\end{equation}
\end{lema}

\begin{proof}
This is a straightforward calculation, using
$\That_{gh} = T_{h^{-1}g^{-1}}$, $p_{h^{-1}} = p_h^t$,
and~\eqref{eq:group-action-on-T}:
\begin{align*}
\det^{1/2}(1 - S\That_{gh})
&= \det^{1/2}(1 - S(\That_h + p_h^{-1}
\That_g (1 - T_h\That_g)^{-1} p_h^{-t}))
\\
&= \det^{1/2}(p_h^{-t}(1 - S\That_h)p_h^t - p_h^{-t} S p_h^{-1}
\That_g (1 - T_h\That_g)^{-1}))
\\
&= \det^{-1/2}(1 - T_h \That_g)\,
\det^{1/2}((1 - S\That_h) p_h^t (1 - T_h\That_g) p_h^{-t}
- S p_h^{-1} \That_g p_h^{-t})
\\
&= \det^{-1/2}(1 - T_h \That_g)\, \phi_h(S)\,
\det^{1/2}(1 - T_h\That_g - p_h^{-t}
(1 - S\That_h)^{-1} S p_h^{-1} \That_g)
\\
&= \det^{-1/2}(1 - T_h\That_g)\,\vf_h(S)\, \vf_g(h\.S).
\tag*{\qed}
\end{align*}
\hideqed
\end{proof}

It follows immediately that
\begin{subequations}
\label{eq:cocycle-F} % (4.41)
\begin{equation}
\mu(g) \mu(h) = c(g,h) \mu(gh),
\label{eq:cocycle-F-defn} % (4.41a)
\end{equation}
where $c(g,h) \in \rU(1)$ is now given by
\begin{align}
c(g,h) &:= c_g c_h c_{gh}^{-1} \det^{1/2}(1 - T_h\That_g)
\nonumber \\
&= \exp\bigl[ i\Arg\det^{+1/2}(1 - T_h\That_g) \bigr].
\label{eq:cocycle-F-formula} % (4.41b)
\end{align}
\end{subequations}
This follows by applying both sides of~\eqref{eq:cocycle-F-defn} to 
any Gaussian $f_S$ under the hypotheses of the previous Lemma. The 
second expression in~\eqref{eq:cocycle-F-formula} comes from
\eqref{eq:explicit-cocycle-B} with the powers of the determinants 
reversed.

\begin{corl} % 4.13
\label{cr:cocycle-relation}
Assume that $g,h,k,gh,hk \in \rSO'_*(V)$. Then the following 
\emph{cocycle relation} is valid:
\begin{equation}
c(g,h)\,c(gh,k) = c(g,hk)\,c(h,k).
\label{eq:cocycle-relation} % 4.42
\end{equation}
\end{corl}

\begin{proof}
On rewriting \eqref{eq:cocycle-F-formula} as
$c(g,h) = c_g c_h c_{gh}^{-1} \vf_g(T_h)$, the desired relation
reduces to 
$$
c_g c_h c_k c_{ghk}^{-1} \vf_g(T_h) \vf_{gh}(T_k)
\overset{?}{=} c_g c_h c_k c_{ghk}^{-1} \vf_g(T_{hk}) \vf_h(T_k).
$$
On recalling that $T_{hk} = h\.T_k$, it follows from 
\eqref{eq:two-scalar-factors} that
$$
\vf_g(T_h) \vf_{gh}(T_k)
= \det^{-1/2}(1 - T_h\That_g)\, \vf_g(T_h) \vf_h(T_k)\, \vf_g(T_{hk})
$$
and, from~\eqref{eq:scalar-factor-F}, the formula
$\vf_g(T_h) = \det^{+1/2}(1 - T_h\That_g)$ establishes the 
cocycle relation.
\end{proof}

Now, the calculations in the proofs of Lemma~\ref{lm:two-scalar-factors}
and Corollary~\ref{cr:cocycle-relation} hold also for the metaplectic 
group, on reversing the powers of the determinants and taking 
$S \in \sD(V)$; $g,h,k \in \rSp(V)$, without restriction. Therefore, 
the cocycle relation \eqref{eq:cocycle-relation} holds also for the 
cocycles \eqref{eq:explicit-cocycle-B} of the metaplectic 
representation.

\marker
It remains to extend $\mu(g)$ to the full Fock space $\sF(V)$ and to
verify the intertwining relation \eqref{eq:full-entwine-fer} there.
The key is to identify the field operator $\psi(v)$ with the image
under~$\mu$ of a certain orthogonal transformation~\cite{Rhea}. For
$v \in V$ with $d(v,v) = 1$, define $r_v \in \rO(V)$ by
\begin{equation}
r_v(w) := 2d(v,w) v - w.
\label{eq:neg-reflection} % (4.43)
\end{equation}
Here $-r_v$ is the reflection in the real hyperplane orthogonal
to~$v$. In particular, $r_v^2 = 1$. Notice also that, if
$g \in \rO(V)$, then
$$
g r_v g^{-1} : w \mapsto 2d(v, g^{-1}w) gv - w
= 2d(gv, w) gv - w = r_{gv}(w),
$$
so that $r_{gv} = g r_v g^{-1}$. We now define, simply:
\begin{equation}
\mu(r_v) := \psi(v).
\label{eq:mu-extended} % (4.44)
\end{equation}

By this means, the domain of $\mu$ has been extended to include some 
improper orthogonal transformations (with $k = 1$). We can also put
$$
\mu(g r_v) := \mu(g) \mu(r_v)  \word{for all}  g \in \rSO'_*(V),
$$
and take $c(g,r_v) := 1$. Checking that these partial definitions are
well defined and that the cocycle
relations~\eqref{eq:cocycle-relation} continue to hold requires
requires some bookkeeping, which we defer to
Appendix~\ref{app:pin-repn}. There we find that 
$$
\mu(g) \mu(r_v) = \mu(gr_v) \mu(g)
\word{for} g \in \rSO_*(V), \ v \in V.
$$
To define the action of $\mu(g)$ on $\sF_1(V)$, it helps to note that
each operator $\psi(v)$, with $v \neq 0$, exchanges the subspaces
$\sF_0(V)$ and $\sF_1(V)$, since $\psi(v)^{-1} = d(v,v)^{-1} \psi(v)$.
Therefore, we can now define $\mu(g)$ on~$\sF_1(V)$ by setting
\begin{equation}
\mu(g) [\psi(v)\,f_S] := \psi(gv)\, \mu(g) f_S
\label{eq:entwine-two} % (4.45)
\end{equation}
which is consistent with \eqref{eq:full-entwine-fer}. The intertwining
$\mu(g) \psi(v) = \psi(gv) \mu(g)$ holds also on~$\sF_1(V)$: see
Proposition~\ref{pr:entwine-three} below.

\goodbreak %%%

The relation \eqref{eq:entwine-two} shows also that each $\mu(g)$ 
leaves the subspace $\sF_1(V)$ invariant, and indeed acts irreducibly 
on~$\sF_1(V)$. The same is true for all $g \in \rSO(V)$, see the 
appendix for the general case. Thus the spin representation of 
$\rSO(V)$ decomposes as the direct sum of two irreducible 
subrepresentations, just like the metaplectic representation.

%=====================================================================

\appendix

\section{The pin representation of the full orthogonal group} % A
\label{app:pin-repn}

To obtain the non-Gaussian out-vacua stated in
Proposition~\ref{pr:defect-vacuum}, one must examine all the strata 
of the orthogonal group. Begin with the negative-refection $r_v$
of~\eqref{eq:neg-reflection}. One checks easily that 
$d(r_v(u), r_v(w)) = d(u,w)$.
%%% Adrián: please make this check!

Let $V'$ be a $2$-dimensional complex subspace of~$V$ generated by
two vectors $u,v \in V$ such that $\braket{u}{u} = \braket{v}{v} = 1$
and $\braket{u}{v} \neq 0$. Put
$s := r_u r_v \in \rSO(V') \subset \rSO(V)$. The subspace $V'$ has an
orthonormal basis $\{v,v'\}$, so that $u = \al v + \bt v'$ with
$\al \neq 0$, $\al\bar\al + \bt\bar\bt = 1$.  Using
$p_s(w) = \half(s(w) - i\,s(iw))$, one may compute $s$ in the
$\{v,v'\}$ basis:
\begin{alignat*}{2}
p_s(v) &= \al(\al v + \bt v'),
& p_s(v') &= \al(-\bar\bt v + \bar\al v'),
\\
q_s(v) &= -\bt(\bar\bt v - \bar\al v'), \quad
& q_s(v') &= -\bt(\al v + \bt v').
\end{alignat*}
The linear map $p_s$ is invertible on $V'$ since $\al \neq 0$, and 
its inverse is given by $v \mapsto (\bar\al v - \bt v')/\al$,
$v' \mapsto (\bar\bt v + \al v')/\al$. From this it follows at once 
that
\begin{equation}
T_s v = (\bt/\bar\al) v',  \qquad  T_s v' = -(\bt/\bar\al) v,
\label{eq:little-T} % (A.1)
\end{equation}
consistent with the skewsymmetry of~$T_s$.

On the complementary subspace $V \ominus V'$, it is clear that 
$p_{r_u} = p_{r_v} = -1$, $q_{r_u} = q_{r_v} = 0$; and so
$T_s = 0$ on that subspace. Therefore, $p_s$ is invertible on~$V$, 
and also
\begin{equation}
f_{T_s} = \Om + \braket{v'}{T_s v}\, v' \w v
= \braket{u}{v}^{-1} (\braket{u}{v}\,\Om + u \w v),
\label{eq:little-Gaussian} % (A.2)
\end{equation}
with $\|\braket{u}{v}\,\Om + u \w v\| = 1$ in the norm of~$\sF(V)$. 
Thereby, since $c_s > 0$:
$$
\mu(s)\,\Om = c_s f_{T_s}
= \exp(-i\arg\braket{u}{v})\, (\braket{u}{v}\,\Om + u \w v).
$$

Now, from \eqref{eq:mu-extended} together with
\eqref{eq:fermion-splitting} and~\eqref{eq:birth-and-death}, it
follows that
\begin{equation}
\mu(r_u) \mu(r_v)\,\Om = \psi(u) \psi(v)\,\Om 
= \braket{u}{v} \Om + u \w v,
\label{eq:push-two} % (A.3)
\end{equation}
and thus $\mu(r_u) \mu(r_v)\,\Om = c(r_u,r_v) \mu(r_u r_v)\,\Om$, on
introducing
$$
c(r_u,r_v) := \exp(i\arg \braket{u}{v}).
$$
These relations remain valid if $\braket{u}{v} = 0$, in which case
$\mu(s)\,\Om = u \w v$.

Note in passing that $[\mu(r_u), \mu(r_v)]_+\,\Om = 2d(u,v)\,\Om$
follows from~\eqref{eq:push-two}, validating the definition
\eqref{eq:mu-extended} of~$\mu(r_v)$.

\begin{lema} % A.1
\label{lm:two-reflections}
If $u,v$ are two unit vectors in~$V$ and if $\Psi \in \sF_0(V)$, then
$$
\mu(r_u) \mu(r_v)\,\Psi = c(r_u,r_v) \mu(r_u r_v)\,\Psi.
$$
\end{lema}

\begin{proof}
For the proof, we may assume $\braket{u}{v} \neq 0$. Since
$\braket{u}{u} = \braket{v}{v} = 1$, one checks that
$$
\psi(u) \psi(v) \bigl( \braket{v}{u} \Om - u \w v \bigr) = \Om,
$$
and thus, with $s = r_u r_v$, $s^{-1} = r_v r_u$:
$$
\mu(s) \bigl( \braket{v}{u} \Om - u \w v \bigl)
= \braket{v}{u} \mu(s) f_{\That_s}
= \braket{v}{u} c_s^{-1} \Om = \exp(-i\arg\braket{u}{v})\,\Om.
$$
Therefore, $\mu(r_u) \mu(r_v) = c(r_u,r_v) \mu(r_u r_v)$ holds on the
subspace $\bC \Om + \bC(u \w v) \subset \sF_0(V)$. The complementary
subspace is generated by Gaussians $f_T$ where $Tu = Tv = 0$. For such
$f_T$, one finds that
$$
\psi(u) \psi(v) f_T = \braket{u}{v} f_T + u \w v \w f_T,
$$
in view of~\eqref{eq:kill-Gaussian-F} with $e_1 = u$ or~$v$; whereas
$\mu(s) f_T = c_s\,\vf_s(T) f_{s\.T}$ with $\vf_s(T) = 1$, and one 
checks that $s\.T = T_s + T$. 
%%% Adrián: please make this check!
%%% [Evaluate $s\.T$ from (2.14), on $V'$ and on $V \ominus V'$]

Then \eqref{eq:little-Gaussian} entails
$$
\mu(s) f_T = \exp(-i\arg\braket{u}{v})\,
\bigl( \braket{u}{v}\,\Om + u \w v \bigr) \w f_T.
$$
This gives the desired result, since each $\Psi \in \sF_0(V)$ is a
linear combination of Gaussians.
\end{proof}

\begin{prop} % A.2
\label{pr:entwine-three}
If $g \in \rSO'_*(V)$ and if $u,v \in V$ are unit vectors, then the
following relation holds for all $\Psi \in \sF_0(V):$
\begin{equation}
\mu(g) \psi(u) \psi(v)\,\Psi = \psi(gu) \psi(gv) \mu(g)\,\Psi.
\label{eq:entwine-three} % (A.4)
\end{equation}
\end{prop}

\begin{proof}
Since $\psi(u) \psi(v) = \mu(r_u) \mu(r_v) = c(r_u,r_v) \mu(r_u r_v)$
as operators on $\sF_0(V)$, it suffices to show that
$$
c(r_u,r_v) \mu(g) \mu(r_u r_v)
= c(r_{gu},r_{gv}) \mu(r_{gu}r_{gv}) \mu(g)  \word{over}  \sF_0(V).
$$
Now $r_{gu} r_{gv} = g r_u r_v g^{-1}$, or equivalently
$r_{gu} r_{gv} g = g r_u r_v$ in $\rSO(V)$, so it is enough to
establish the following cocycle identity, extending
\eqref{eq:cocycle-relation}:
\begin{equation}
c(r_u, r_v)\, c(g, r_u r_v) = c(r_{gu},r_{gv})\, c(r_{gu}r_{gv}, g),
\label{eq:cocycle-identity} % (A.5)
\end{equation}
for $g,u,v$ satisfying the stated hypotheses.

Put $s := r_u r_v$; then the left-hand side
of~\eqref{eq:cocycle-identity} equals
\begin{equation}
\exp(i\arg\braket{u}{v}) c_g c_s c_{gs}^{-1}
 \det^{1/2}(1 - T_s \That_g).
\label{eq:two-cocycles} % (A.6)
\end{equation}
Assume that $\braket{u}{v} \neq 0$ and let $\{v,v'\}$ be an
orthonormal basis for $\bC u + \bC v$, as before. The Pfaffian
expansion \eqref{eq:Gaussian-pairing-F} of the last factor has just
two terms:
$$
\det^{1/2}(1 - T_s \That_g)
= 1 + \braket{\That_g v}{v'}\, \braket{v'}{T_s v} 
$$
and so \eqref{eq:two-cocycles} simplifies to
\begin{align*}
c_g c_{gs}^{-1} \braket{u}{v}
\bigl( 1 + \braket{\That_g v}{v'}\, \braket{v'}{T_s v} \bigr)
&= c_g c_{gs}^{-1} \bigl( \braket{u}{v}
+ \braket{\That_g v}{\bt v'} \bigr)
\\
&= c_g c_{gs}^{-1}\bigl( \braket{u}{v} + \braket{\That_g v}{u} \bigr).
\end{align*}
where the last equality uses \eqref{eq:little-T}, namely
$\braket{v'}{T_s v} = \bt \braket{u}{v}^{-1}$. Thus
$$
c(r_u, r_v)\, c(g, r_u r_v)
= c_g c_{gs}^{-1}\bigl( \braket{u}{v} + \braket{\That_g v}{u} \bigr),
$$
which remains valid when $\braket{u}{v} = 0$, too.

\goodbreak %%%

In like manner, with $s' := r_{gu}r_{gv}$, the right-hand side 
of~\eqref{eq:cocycle-identity} is
$$
\exp(i\arg\braket{gu}{gv})\, c_g c_{s'} c_{s'g}^{-1}\,
\det^{1/2}(1 - T_g\That_{s'}), 
$$
which simplifies to
$$
c_g c_{gs}^{-1} \braket{gu}{gv}
\bigl( 1 + \braket{\That_{s'}gu}{u'}\, \braket{u'}{T_g gu} \bigr),
$$
where $u'$ is a unit vector orthogonal to~$gu$ in the subspace
$\bC gu + \bC gv$. One checks that in this case
\eqref{eq:little-T} produces
$$
\braket{gu}{gv} 
\bigl( 1 + \braket{\That_{s'}gu}{u'}\, \braket{u'}{T_g gu} \bigr)
= \braket{gu}{gv} + \braket{gv}{T_g gu}.
$$
The result now follows from the next calculation, that uses the 
skewsymmetry of~$\That_g$:
\begin{align*}
\MoveEqLeft{\braket{gu}{gv} + \braket{gv}{T_g gu} 
= \braket{gu}{gv} + \braket{p_gv + q_gv}{q_g(u + p_g^{-1}q_gu)}}
\\
&= \braket{gu}{gv} + \braket{u - \That_gu}{q_g^t(p_gv + q_gv)}
\\
&= \braket{u}{(p_g^t + q_g^t)(p_gv + q_gv)} + \braket{q_gu}{gv}
+ \braket{\That_gu}{p_g^tq_gv - v + p_g^tp_gv}
\\
&= \braket{u}{v} - \braket{\That_gu}{v}
+ \braket{q_gu}{gv} + \braket{p_g\That_gu}{gv}
\\
&= \braket{u}{v} - \braket{\That_gu}{v}
= \braket{u}{v} + \braket{\That_gv}{u}.
\tag*{\qed}
\end{align*}
\hideqed
\end{proof}

Combining \eqref{eq:entwine-two} with \eqref{eq:entwine-three}, it is
now clear that
\begin{equation}
\mu(g)\,[\psi(u)\,\Phi] = \psi(gu)\,[\mu(g)\,\Phi]
\label{eq:easy-entwiner} % (A.7)
\end{equation}
both when $\Phi = \Psi \in \sF_0(V)$ and when 
$\Phi = \psi(v)\,\Psi \in \sF_1(V)$. The intertwining relation 
\eqref{eq:full-entwine-fer} is thereby established whenever 
$g \in \rSO_*(V)$ and $v \in V$.

Next, replace $\rSO_*(V)$ by group elements of the form $gr_u$, with
$g \in \rSO_*(V)$, $u \in V$. The relations
$$
\psi(u) \psi(v) = 2\,d(u,v) - \psi(v) \psi(u) = \psi(r_uv) \psi(u)
$$
allow us to rewrite \eqref{eq:easy-entwiner} in the form:
$$
[\mu(g) \psi(u)] \psi(v)\,\Phi = \mu(g) \psi(r_uv) \psi(u)\,\Phi
= \psi(gr_u v) [\mu(g) \psi(u)]\,\Phi,
$$
which entails $\mu(gr_u) \psi(v) = \psi(gr_uv) \mu(gr_u)$ as 
operators on~$\sF(V)$. Therefore, \eqref{eq:full-entwine-fer} also 
when $g$ is replaced by~$gr_u$. Since $g r_u g^{-1} = r_{gu}$, it 
also holds for group elements of the form $r_v g$, always with
$g \in \rSO'_*(V)$.

\marker
We may at last consider the general case with $g \in \rO(V)$.

Assume that $\dim \ker p_g^t = \dim \ker p_g = k > 0$. Then we can 
factorize $g = rh$ as in subsection~\ref{ssc:spin-repn}, with 
$h \in \rSO_*(V)$ and $r := r_1 \cdots r_k$ a product of $k$
negative-reflections determined by an \textit{oriented}
orthonormal basis $\{e_1,\dots,e_k\}$ for $\ker p_g^t$. Now define
$$
\mu(g) := \psi(e_1) \cdots \psi(e_n)\,\mu(h).
$$
Then $\mu(g)\,\Om = e_1 \wyw e_n \w f_{T_h}$, as expected. It remains 
only to check that the cocycle $c$ may be extended compatibly to the 
general case. For instance, since $c(r_v,h) = 1$, one must take
$c(r_u, r_vh) := c(r_u,r_v)\,c(r_ur_v, h)$; and so on.

This ``pin representation'' of $\rO(V)$ restricts to the subgroup
$\rSO(V)$ as a direct sum of two irreducible subrepresentations on
$\sF_0(V)$ and on $\sF_1(V)$, respectively. (They are inequivalent:
they extend to the two inequivalent unirreps of $\rSpinc(V)$, see for 
instance \cite[Lemma~5.11]{Polaris}.)

\subsection*{Acknowledgments}
We thank José M. Gracia-Bondía for much helpful advice on these and
related matters. We acknowledge support from the Vicerrectoría de
Investigación of the Universidad de Costa~Rica.

%%%%%%%%%%%%%%%%%%%%%%%%%%%%%%%%%%%%%%%%%%%%%%%%%%%%%

% \newpage %%%

\bigskip   %%% ***

\end{document}